\documentclass[leqno]{amsart}
\usepackage{amsfonts,amssymb,amsmath,amsgen,amsthm,datetime}
\usepackage{hyperref}
\usepackage[scale=0.70]{geometry}

\theoremstyle{plain}
\newtheorem{theorem}{Theorem}[section]

\newtheorem{lemma}[theorem]{Lemma}

\newtheorem{proposition}[theorem]{Proposition}

\theoremstyle{remark}
\newtheorem{remark}[theorem]{Remark}

\numberwithin{equation}{section}

\usepackage{latexsym,amssymb}
\usepackage{graphicx}
\usepackage{xcolor}
\definecolor{Nico}{RGB}{220,20,60}
\definecolor{Clo}{RGB}{0,128,128}
\definecolor{Alain}{RGB}{255,140,0}

\def\Im{\;\textrm{Im}\;}
\def\C{{\mathbb C}}

\def\N{{\mathbb N}}
\def\R{{\mathbb R}}

\def\e{{\rm e}}

\def\eps{\varepsilon}
\def\op_#1{\mathrel{\mathop{{\rm op}_{#1}}}}

\def\build#1_#2^#3{\mathrel{
\mathop{\kern 0pt#1}\limits_{#2}^{#3}}}

\def\td_#1,#2{\mathrel{
\mathop{\build\longrightarrow_{#1\rightarrow #2}^{}}}}

\def\limsup_#1,#2{\mathrel{
\mathop{\build{\rm limsup}_{#1\rightarrow#2}^{}}}}

\def\liminf_#1,#2{\mathrel{
\mathop{\build{\rm liminf}_{#1\rightarrow#2}^{}}}}

\def\eps{\varepsilon}

\def\1{{\bf 1}}
\def\0{{\bf 0}}

\def\op{{\rm op}}

\title{Edge modes generated by intersection between 3 energy bands}

\author[C. Fermanian Kammerer]{C. Fermanian Kammerer}
\address[C. Fermanian Kammerer]{Univ Angers, CNRS, LAREMA, F-49000 Angers, France}
\email{clotilde.fermanian@univ-angers.fr}

\author[N. Frantz]{N. Frantz}
\address[N. Frantz]{Univ Angers, CNRS, LAREMA, F-49000 Angers, France}
\email{nicolas.frantz@univ-angers.fr}

\date{\today}

\begin{document}

\begin{abstract}
In this contribution, we investigate 2-dimensional model problems of three coupled equations. We assume that the underlying Hamiltonian presents a symmetric intersection between three eigenvalues of Dirac's type with a mass function that vanish along a curve. We investigate the existence of edge modes for these models: such functions are  solutions of the semiclassical associated evolution problem that are asymptotic to a coherent state with zero energy and  break the correspondence principle by propagating inside the crossing set and not along a classical trajectory. 
\end{abstract}

\maketitle

\tableofcontents

\section{Introduction}

Topological insulators are specific materials that behave like insulators in their bulks, while currents propagate on their boundaries~\cite{H88,RH}. These special states are called edge modes. 
The  mathematical modeling of materials resorts to solid state physics and the existence of edge modes is understood as a consequence of conical intersections between Bloch modes~\cite{FLW_16}. In dimension~$2$, the underlying model is the so-called Dirac cone: near a generic intersection point~$\xi^*$, the Bloch modes 
behave like $E ^*\pm |\xi-\xi^*|$, with $\xi\in\R^2$. Therefore, the simplest local model looks like a Dirac  model with a $2\times 2$ matrix-valued Hamiltonian with eigenvalues $\pm|\xi|$. One modifies the model to include the hypothesis that these conical intersections appear along an interface 
\[
E:=\{x\in\R^2\,|\,m(x)=0\}
\]
between the two topological insulators.  
Then, the eigenvalues are distorted into $\pm\sqrt{m(x)^2+|\xi|^2}$. Up to conjugation with a unitary matrix and a sign factor, there exists a unique smooth matrix that depends linearly on $m(x)$ and $\xi$, and has eigenvalues $\pm\sqrt{m(x)^2+|\xi|^2}$: it is the matrix 
$\displaystyle{\begin{pmatrix}m(x) & \xi_1\mp i\xi_2\\ \xi_1\pm i\xi_2 & -m(x)\end{pmatrix}}$,
which is associated with the evolution equation 
\begin{equation}\label{eq:2by2}
ih\partial_t \psi^h= \begin{pmatrix}m(x) & hD_{x_1}\mp ihD_{x_2}\\ hD_{x_1}\pm ihD_{x_2}& -m(x)\end{pmatrix}\psi^h,\;\;\psi^h(0)=\psi^h_0\in L^2(\R^2,\C^2),
\end{equation}
where $D_{x_j}=\frac 1i \partial_{x_j}$ (this claim is discussed in Appendix~\ref{app:claim}). 
\smallskip 

An abundant literature has been devoted to the analysis of the properties of these 2 by 2 models. 
Under the condition 
\begin{equation}\label{eq:min_m}
\min_{x\in\R^2} |\nabla m(x)|=c_0>0, 
\end{equation} 
it has been proved in~\cite{BBDFLW} that there exist approximate solutions of the form  \begin{equation}\label{def:edgestate}
(t,x)\mapsto \frac {1}{\sqrt h} \varphi \left(t,\frac{x-x_t}{\sqrt h}\right)\;\;
\mbox{with}\;\;m(x_t)=0, \;\;|\dot x_t|_{\R^2}=1,
\end{equation}
and $\varphi(t,\cdot)\in\mathcal S(\R^2,\C^2)$. Since then, similar results have been established in the presence of a magnetic field~\cite{BBD23}, and a complete study of the propagation along the interface in terms of Wigner measures for families of initial data that are bounded in $L^2(\R^2,\C^2)$ without a coherent state structure has been performed in~\cite{vacelet}. More generally, starting from a 2 by 2 matrix-valued Hamiltonian, conditions have been given for the existence of edge modes, together with their characterization via a normal form procedure in~\cite{D22}. 
\smallskip 

Wave packets of the form \eqref{def:edgestate} belong to the class of coherent states introduced in the 80s~\cite{Hagedorn_80,coro_book}. They have the property of being microlocalized   at the scale $\sqrt h$, simultaneously in position and impulsion, saturating the Heisenberg's uncertainty principle.
\smallskip 

Our aim here is to address a similar question when the Bloch structure of the material involves a symmetric conical intersection  between three bands, assuming that they behave locally, near a point $k^*$, like $E^*$, $E^*+|k-k^*|$ and $E^*-|k-k^*|$.
Following the same lines of thoughts as above, we are  interested in  $3$-bands models related with $2$-d topological insulators, of the form
\[
\widehat H=m(x) \Theta_0 + h D_{x_1} \Theta_1+ hD_{x_2}\Theta_2,
\]
where $\Theta_0$, $\Theta_1$ and $\Theta_2$ are constant $3\times 3$ self-adjoint matrices which are linearly independent.
%and $m(x)=0$ is the equation between the two distinct topological insulators. 
We assume $m$ belongs to $\mathcal{C}^\infty(\R^2,\R)$ and  satisfies~\eqref{eq:min_m},
and  that the eigenvalues of the symbol of $\widehat H$
\begin{equation}\label{def:H}
H(x,\xi)=m(x) \Theta_0 + \xi_1 \Theta_1+ \xi_2\Theta_2,
\end{equation}
are the functions 
\[
-\sqrt{m(x)^2 +\xi_1^2+\xi_2^2},\; 0,\; \sqrt{m(x)^2 +\xi_1^2+\xi_2^2}.
\]
They cross simultaneously above the points of the interface $E$ that have zero momenta, i.e. the points of the set  
\[
\mathcal E=E\times \{(0,0)\}\subset T^*\R^2=\R^2_x\times \R^2_\xi.
\]
Hamiltonians given by $3\times 3$ self-adjoint matrix symbols with linear dependence in the momentum variables and a scalar mass term appear in several independent settings involving three-band conical degeneracies. In particular, shallow-water-type models exhibiting bulk–edge correspondence provide concrete realizations of such effective $3\times 3$ structures \cite{B26,Tau19}.

Our aim is to analyze whether there exist edge modes and to characterize them. Thus,  we consider the equation 
\begin{equation}\label{eq:dirac}
ih\partial_t \psi^h=\widehat H\psi^h,\;\;\psi^h(0)=\psi^h_0\in L^2(\R^2,\C^3),
\end{equation}
and look for approximate solutions of the form~\eqref{def:edgestate} with $\varphi(t,\cdot)\in\mathcal S(\R^2,\C^3) $ for all $t\in\R$. 
% Such solutions are called {\it edge states}. 
Our aim is to find  all the $3\times 3$ self-adjoint matrices $H$ as in~\eqref{def:H} 
with the appropriate eigenvalues,  to determine conditions on $\Theta_0$, $\Theta_1$ and $\Theta_2$ for the existence of edge states, and to characterize their propagation.

\subsection{Three bands models and edge states}

Consider 
 coherent states of the form  
\begin{equation}\label{def:wp}
\operatorname{WP}_{x_t}(\varphi(t,\cdot))(x):=\frac 1{\sqrt h} \varphi\left(t,\frac{x-x_t}{\sqrt h}\right),\; 
\;\varphi\in C^0(\R,\mathcal S(\R^2,\C^3)),
\end{equation}
with $t\mapsto x_t$ a curve included in the set $m^{-1}\left(\lbrace 0\rbrace\right)$. The condition on $x_t$ implies that 
\begin{equation}\label{def:curve}
\dot x_t = \rho_t\frac{\nabla m(x_t)^\perp}{|\nabla m(x_t)|},\;\;\forall t\in\R,
\end{equation}
for some smooth function $t\mapsto \rho_t\in\R$, and where we have set $(a,b)^\perp=(-b,a)$.  We will say that the solution $\psi^h(\cdot)$ of~\eqref{eq:dirac} is an edge state if for all $T>0$, there exist $C>0$, $\alpha>0$, and $\varphi$ such that 
\[
\sup_{t\in[0,T]}\| \psi^h(t)-\operatorname{WP}_{x_t}(\varphi(t,\cdot))\|_{L^2(\R^2,\C^3)}\leq Ch^{\alpha} .
 \]
In order to have states with the physical interest, we will also ask that $\rho_t$ is not identically zero. Such {\it degenerate edge states} do not exist with $2$ by $2$ models, unlike to the $3$ by $3$ situation we study here.
We will call {\it generalized edge states} any function that concentrates microlocally at all time on $\mathcal E$, this concentration being measured in terms of semi-classical measures. Such solutions can be decomposed as a superposition of waves that have different speeds along the interface~$E$, including slower ones than the edge state itself. This is described for models with two bands in~\cite{vacelet} (see Remark~\ref{rem:wigner}).
\smallskip

We will distinguish three cases, depending on the relations the matrices have with each other. When $A,B\in\C^{3,3}$, we set $[A,B]=AB-BA$. We associate with $\Theta_0$, $\Theta_1$ and $\Theta_2$ the  
subspace $\mathcal V$ of~$\C^{3,3}$ defined by 
\[
\mathcal V={\rm Span}
\bigl([\Theta_i,\Theta_j],\;\;\Theta_i,\;\;0\leq i,j\leq 2\bigr)
.\]
Our main result establishes that the dimension of the vector space~$\mathcal V$ is one of the crucial elements that characterize the existence of edge states. 

Another important feature concerns the properties of the kernel of the matrix~$H(x,\xi)$. Spectral projectors are homogeneous with respect to the parameter $m(x)$, $\xi_1$ and $\xi_2$ and so they are smooth if and only if they are constant. In this case, by conjugating with a unitary matrix, we reduce the Hamiltonian to a $2\times 2$ matrix that is well-known. So, having or not having constant vectors in the kernel of $H(x,\xi)$ will also be an important characteristic feature.

\begin{theorem}\label{theo:main}

Let $H(x,\xi)$ be the Hamiltonian  given by~\eqref{def:H}, with eigenvalues 
\[
-\sqrt{m(x)^2+\xi_1^2+\xi_2^2}, \;\;0,\;\; \sqrt{m(x)^2+\xi_1^2+\xi_2^2}.
\]
Then we have the following alternative:
\begin{enumerate}
\item[(i)] If  ${\rm dim}\mathcal V=3$ and if there exists a constant vector in the kernel of $(x,\xi)\mapsto H(x,\xi)$, there exists a unique family of edge states
associated with a trajectory $x_t$ that satisfies~\eqref{def:curve} with 
\[
\rho_t=-{\rm sgn}\left(\frac 1i{\rm Tr} \left(\bigl[ \Theta_1,\Theta_2\bigr]\Theta_0\right)\right).
\]
Moreover, there exists a degenerate edge state. 
\item[(ii)] If  ${\rm dim}\mathcal V=3$, and if there is no  constant vector in the kernel of $(x,\xi)\mapsto H(x,\xi)$, there exists a unique family of edge states associated with a trajectory $x_t$ that satisfies~\eqref{def:curve} with 
\[
\rho_t={\rm sgn}\left(\frac 1i{\rm Tr} \left(\bigl[ \Theta_1,\Theta_2\bigr]\Theta_0\right)\right).
\]
\item [(iii)] If $\dim \mathcal V=5$, there exist two families of edge states associated with two different trajectories~$x_t$ that satisfy~\eqref{def:curve}, one  with  $\rho_t=+1$ and the other with  $\rho_t=-1$. 
Moreover, there exists a degenerate edge state. 
\item [(iv)] If $\dim \mathcal V=6$, there is only a degenerate edge state.
\end{enumerate}
\end{theorem}

Similarly with what happens for Hamiltonians given by $2\times 2$ matrices, the edge states (when they exist) are associated with initial data that are close (with $h$) to precise subspace in $L^2(\R^2,\C^3)$.

In contrast to the situation for $2\times2$ matrices, where edge states always exist and the direction of propagation along $\mathcal{E}$ is fixed, the context of $3\times3$ matrices is richer. There may exist edge states, with a fixed direction of propagation, or none of them, or two propagating along opposite directions on $\mathcal E$. In the latter case, the initial data that generate edge states form a two-dimensional subspace of $L^2(\mathbb{R}^2,\C^3)$.

Before describing examples and explaining the main ideas of the proof, let us mention  that the propagation of coherent states through systems of equations presenting crossing has attracted a lot of attention both in the mathematical \cite{Hag94,HJ2,FL1,FGH,Curely} and theoretical chemistry communities \cite{Tully_Preston,DYK} for example. These works were motivated by the description of molecular dynamics in  Born-Oppenheimer approximation, leading to more general questions (see for example the book \cite{Carles_book} which addresses non linear propagation of coherent states). However, the systems considered for molecular dynamics are usually generic in the sense of \cite{CdV03,CdV04}: the classical trajectories that one associates with the system reach (and may pass through) the crossing set. These trajectories are Hamiltonian trajectories of the eigenvalues of the principal symbol of the operator leading the system of PDEs. For the system \eqref{eq:dirac}, the quantity  $m(x)^2+|\xi|^2$ is constant along the Hamiltonian curves associated with the functions $ \sqrt{m(x)^2+|\xi|^2}$ and  $-\sqrt{m(x)^2+|\xi|^2}$, and none of the classical trajectories of the system reaches the set $\mathcal E$, leading to a completely different setting to the (generic) one of   
\cite{FG1} (see also \cite{FG2}) where a electric potential generating a non zero electric field is added in the equation \eqref{eq:2by2}.

\subsection{Examples}\label{sec:Ex}
We consider the following matrices:
\begin{align*} 
& \Gamma_{\rm JT}(\sigma,z,\zeta)=\begin{pmatrix}
    \sigma & z-i \zeta  & 0\\
   z+i \zeta & -\sigma  & 0\\
    0 & 0 & 0
    \end{pmatrix},\;\;
    \Gamma_{\rm PJT}(\sigma,z,\zeta)= \begin{pmatrix}
    \sigma & 0  &  \frac 1{\sqrt 2}( z -i \zeta)\\
   0 & -\sigma  &  \frac 1{\sqrt 2}( z +i \zeta)\\
    \frac 1{\sqrt 2}( z +i \zeta) & \frac 1{\sqrt 2}( z -i \zeta)& 0
    \end{pmatrix},  \\
&
      \Gamma_{\alpha,\gamma}(\sigma,z,\zeta)= \begin{pmatrix}
    \sigma & i\gamma z  & \alpha z +\frac i{\sqrt 2} \zeta\\
   - i\gamma z & -\sigma  & \alpha z +\frac i{\sqrt 2}  \zeta\\
    \alpha z -\frac i{\sqrt 2} \zeta & \alpha z -\frac i{\sqrt 2}  \zeta & 0
    \end{pmatrix},\;\;\alpha\geq 0,\;\;\gamma^2+2\alpha^2=1.
\end{align*}

The indices JT and PJT in the above notation refer to the appellation of Jahn Teller and pseudo Jahn-Teller that are given to such matrices in theoretical chemistry (see~\cite{DYK}). Matrices of this form appear as models of intersections between three energy surfaces in molecular dynamics. Examples of this form have been studied in~\cite{BE_93,CH} in the chemical literature, and in~\cite{Fermanian_Rousse}. In these contexts, Hamiltonians are of the Schr\"odinger type with matrix-valued potential presenting conical intersections of the form $\Gamma_{\rm JT}(x_1,x_2,x_3)$ or $\Gamma_{\rm PJT}(x_1,x_2,x_3)$. In~\cite{Fermanian_Rousse} it was shown that, although these matrices have the same eigenvalues, the dynamics  generated by their crossings are different.
\smallskip 

Let us now give simple examples with the interface  $\mathcal E=\{x_2=\xi_1=\xi_2=0\}$:
\smallskip

\noindent{\it Example 1 based on the Jahn-Teller Hamiltonian:} The Hamiltonian $\Gamma_{\rm JT}(\xi_1,x_2,\xi_2)$ satisfies the characterization of  (i) in Theorem~\ref{theo:main}. The existence of edge states reduces to the case of 2 by 2 matrices and the results of~\cite{BBDFLW} apply. They are given by the functions 
\begin{equation*}\label{eq:em1}
(t,x)\mapsto \frac{1}{\sqrt h} f\left(\frac {x_1-b-t}{\sqrt h}\right)
{\rm e}^{-\frac{x_2^2}{2h}}
\begin{pmatrix}
1 \\ 0\\ 0
\end{pmatrix},\;\; f\in \mathcal S(\R),\;\; b\in\R.
\end{equation*}
In this easy situation, these functions are exact solutions to~\eqref{eq:dirac} when $H(x,\xi)= \Gamma_{\rm JT}(\xi_1,x_2,\xi_2)$. 
Moreover,  there exist degenerate edge states along the direction  $\,^t(0,0,1)$.
\smallskip 

\noindent {\it Example 2 based on the pseudo Jahn-Teller Hamiltonian:} The Hamiltonian $\Gamma_{\rm PJT}(\xi_1,x_2,\xi_2)$ is related to case  (ii) with edge modes of the form 
\begin{equation*}\label{eq:em2}
(t,x)\mapsto \frac{1}{\sqrt h} f\left(\frac {x_1-b-t}{\sqrt h}\right)
{\rm e}^{-\frac{x_2^2}{2h}}
\begin{pmatrix}
1 \\ 0\\ 0
\end{pmatrix},\; f\in \mathcal S(\R),\; b\in\R.
\end{equation*}

\noindent{\it Example 3, the case with two edge modes moving in opposite directions.} Finally, the Hamiltonian 
$\Gamma_{\frac 1{\sqrt 2},0}(\xi_1,x_2,\xi_2)$ satisfies case (iii) in Theorem~\ref{theo:main}, with edge modes given by functions of the form
\[
(t,x)\mapsto \frac{1}{\sqrt h} f\left(\frac {x_1-a-t}{\sqrt h}\right)
{\rm e}^{-\frac{x_2^2}{2h}}
\begin{pmatrix}
1 \\ 0\\ 0
\end{pmatrix}\;\mbox{and}\;
(t,x)\mapsto \frac{1}{\sqrt h} f\left(\frac {x_1-b+t}{\sqrt h}\right)
{\rm e}^{-\frac{x_2^2}{2h}}
\begin{pmatrix}
0 \\ 1\\ 0
\end{pmatrix},
\]
with $f,g\in \mathcal S(\R)$ and $a,b\in\R$.
 With the notation~\eqref{def:wp}, the degenerated edge states are given by the functions 
%\,^t(\phi,-\phi, T\phi)$ 
\[ (t,x) \mapsto \frac 1{\sqrt h} \, ^t\bigl({\rm WP}_{x_0}(\phi),
-{\rm WP}_{x_0}(\phi),
{\rm WP}_{x_0}(T\phi)\bigr),\;\; x_0\in E, 
\]
with $\phi\in \left( L^2(\R_{x_1})\otimes \C(\e^{-\frac{x_2^2}{2}})\right)^\perp$ and $T\phi$ determined by 
$(x_2-\partial_{x_2})T\phi=i\sqrt 2 \partial_{x_1}\phi$.
The Hamiltonian $\Gamma_{\alpha,\gamma} (\xi_1,x_2,\xi_2)$ for other values of $\alpha$ and $\gamma$ corresponds to case (iv) and has no edge states. 

\subsection{Strategy of the proof}\label{subsec:strategy}
Let $\psi^h$ be a normalized solution of \eqref{eq:dirac} and let $(x_t)_{t\in\R}$ be a curve in $E$. The change of unknown function:
\begin{equation}\label{eq:profile}
\psi^h(t,x)=\frac 1{\sqrt h}\varphi^h \left(t,\frac{x-x_t}{\sqrt h}\right),\;\; \varphi^h(t,\cdot)\in L^2(\R^2,\C^3)
\end{equation}
in Equation~\eqref{eq:dirac}
is always possible.
Plugging this ansatz into the equation~\eqref{eq:dirac}, for $t\in\R$ and $y\in\R^2$, we obtain
\begin{align*}
i h {\partial_t}\varphi^h(t,y) +\sqrt h\dot x_t \cdot D_y\varphi^h(t,y)= \left(m(x_t+\sqrt h y)\Theta_0 +\sqrt h D_{y_1}\Theta_1 +\sqrt h D_{y_2}\Theta_2\right)\varphi^h(t,y).
\end{align*}
Assume $\varphi^h$ is in $\mathcal{S}(\R^2,\C^3)$, in view of $m(x_t)=0$, a Taylor expansion gives for $N\in\N$,
\begin{align}
\label{eq:edge_state}
0= &\left(-\dot x_t \cdot D_y+ y\cdot \nabla m(x_t)\Theta_0 + D_{y_1}\Theta_1 + D_{y_2}\Theta_2\right)\varphi^h(t,y)\\
\nonumber
&+ h^\frac{1}{2} \left( -i\partial_ t  +\frac 12 {\rm Hess}\,  m(x_t) y\cdot y \,\Theta_0\right)\varphi^h(t,y)\\
\nonumber
& + \left(\sum_{2\leq j\leq N} h^\frac j2 \frac 1 {j!} d^j m(x_t) [y]^j\right)\Theta_0\varphi^h(t,y) + h^{\frac {N+1}2}\mathcal O\left(\left\||y|^{N+1}\varphi^h(t,\cdot)\right\|_{L^2(\R^2,\C^3)}\right).
\end{align}
Looking for an approximation of the form $\varphi^h=\sum_{0\leq \ell\leq N} h^{\frac \ell 2}\varphi_\ell$, we deduce that  the function $\varphi_0$ must be chosen as a function of the kernel of operator~$\widehat P_t$ given by 
\begin{equation}\label{def:Pt}
\widehat P_t=- \dot x_t\cdot D_y+y\cdot \nabla m(x_t) \Theta_0 + D_ {y_1}\Theta_1 + D_{y_2}\Theta_2.
\end{equation}
As a consequence, a necessary condition for the existence of edge modes is that the operator $\widehat P_t$ has non zero vectors in its kernel, otherwise $\varphi=0$ (see Section~\ref{sec:Edge}). Note that the eigenvalues of the matrix  $y\cdot \nabla m(x_t) \Theta_0 +  {\eta_1}\Theta_1 + {\eta_2}\Theta_2$  are $0$, $\sqrt{(y\cdot \nabla m(x_t))^2+|\eta|^2}$ and $-\sqrt{(y\cdot \nabla m(x_t))^2+|\eta|^2}$ (see Remark~\ref{rem:eigenvalue_abc}). 
\smallskip 

Using the notation of~\cite{BBDFLW}, we set 
\[
r_t=|\nabla m(x_t)|,\;\;\nabla m(x_t)=r_t \vec e(t),\;\;\vec e(t)=(-\sin \theta_t,\cos\,\theta_t),\;\;\vec e(t)^\perp=-(\cos\,\theta_t,\sin\,\theta_t).
\]
Fixing the time $t\in\R$, we rewrite the operator $\widehat P_t$ as
\[
\widehat P_t= \rho_t \,\vec e(t)^\perp\cdot D_y+r_t\,y\cdot \vec e(t) \Theta_0 + D_{y_1}\Theta_1+ D_{y_2}\Theta_2.
\]
Using the orthogonal change of variable 
\begin{equation}\label{def:coordzs}y\mapsto (y\cdot \vec e(t),y\cdot\vec e(t)^\perp)=(z,s),
\end{equation} the operator $\widehat P_t$ is unitarily equivalent to the operator 
\[
 \rho_t \,D_s+r_t\,z
 \Theta_0 + D_{s}\Theta_1(t)+ D_{z}\Theta_2(t)
\]
with 
\begin{align}\label{eq:Theta}
&\Theta_1(t)=-\cos\,\theta_t \Theta_1 -\sin \,\theta_t\Theta_2 \;\;\mbox{and}\;\; \Theta_2(t)=-\sin\,\theta_t \Theta_1 + \cos\,\theta_t\Theta_2.
\end{align}
Performing  the Fourier transform in variable~$s$, we see that $\widehat P_t$ has non zero vectors in its kernel if and only if it is so for the operator 
\begin{align}\label{eq:gamma_hat}
&\rho_t \sigma + \Gamma_t(\sigma,z,D_z)\;\;\mbox{with}\;\;  \Gamma_t(\sigma,z,D_z)= r_t z\,\Theta_0 + \sigma\Theta_1(t) + D_ {z}\Theta_2(t),
\end{align}
and  where $\sigma$ is the dual variable of $s$.
The symbol of the operator $ \Gamma_t(\sigma,z,D_z)$
is a linear function $\Gamma(\sigma,z,\zeta)$ of the variables $(\sigma,z,\zeta)$, equal to the matrix $y\cdot \nabla m(x_t) \Theta_0 +  {\eta_1}\Theta_1 + {\eta_2}\Theta_2$ through the change of symplectic coordinates $(y;\eta)\mapsto (z,s;\zeta,\sigma)$. In view of  $|\eta|^2=\zeta^2+\sigma^2$ and $y\cdot m(x_t)= r_t z$, we deduce that the  eigenvalues of $ \Gamma_t(\sigma,z,\zeta)$ are  $-\sqrt{(r_tz)^2+\zeta^2+\sigma^2}$, $\sqrt{(r_tz)^2+\zeta^2+\sigma^2}$, and $0$. 
The end of the proof relies on normal forms for the matrix $\Gamma_t(\sigma,z,\zeta)$. We detail these aspects in the next paragraph.

\subsection{Normal forms for 3 by 3 edge states systems}

This classification provides interesting examples of 3 by 3 models of edge states.
\begin{theorem}\label{theo:model_pb}
    Let $\Gamma\in\mathcal L(\R^3, \C^{3\times 3})$ be a linear map valued in the set of self-adjoint matrices with eigenvalues 
    \[
    - \sqrt{\sigma^2+z^2+\zeta^2}, \, 0,\,\sqrt{\sigma^2+z^2+\zeta^2}.
    \]
    Then, there exist $\delta\in\{-1,+1\}$,  a unitary matrix $U$ and a linear symplectic change of variables $\kappa$ on $\R^2=T^*\R$ 
    such that 
    \[
    U\Gamma (\sigma,\kappa(z,\zeta)) U^* =\Gamma_\bullet (\sigma,z,\delta\zeta)
    \]
    where $\bullet$ stands for the indices JT, PJT, $\alpha,\gamma$, $\alpha\geq 0$, $\gamma^2+2\alpha^2=1$ as described in Example \ref{sec:Ex}. 
\end{theorem}

Regarding edge-states, it is the eigenvalues that are of the form  $k\sigma$ for some $k\in\R$ in which we are interested, since taking $\rho_t=k$ will imply that the operator $\widehat P_t$ has some non-zero vectors in its kernel. 
\smallskip

At this stage of the proof, the operator $\Gamma_t(\sigma,z,D_z)$ is unitarily equivalent in $L^2(\R^2)$  to one of the model operators obtained by quantization of the matrices of Theorem~\ref{theo:model_pb}. Indeed, there exists an explicit Fourier integral  operator associated with the linear symplectomorphism $\kappa$ such that 
\begin{equation}\label{eq:reductio}
\Lambda_t ^h U_t \, \Gamma_t (\sigma,z,D_z) \, U_t^* (\Lambda_t^h)^*  =
\Gamma_\bullet(\sigma,r_tz,\delta D_z).
\end{equation}
Here, ``$\bullet$'' stands for one of the indices JT, PJT or ${(\alpha,\gamma)}$. The existence of such an operator is standard (see~\cite{coro_book}) and is recalled in Appendix~\ref{app:OIF}. As a consequence, the analysis of the kernel of the operator $\widehat P_t$ has been transformed into the determination of the spectrum of the model operators. 
\smallskip 

The description of
 the spectrum of the operators 
 \[
 \Gamma_{\rm JT}(\sigma,r_tz,\delta D_z), \;\;\Gamma_{\rm PJT}(\sigma,r_tz,\delta D_z)\;\;\mbox{and} \;\;\Gamma_{\alpha,\gamma}(\sigma,r_tz,\delta D_z)\;\;\mbox{with the condition}\;\;\gamma^2+2\alpha^2=1,
 \]
 as stated in Proposition~\ref{prop:spectral},
 implies the following results which explain the four alternatives of Theorem~\ref{theo:main}:
 \begin{enumerate}
 \item[(i)]
The operators $ \Gamma_{\rm JT}(\sigma,r_tz,\delta D_z)$, $ \Gamma_{\rm PJT}(\sigma,r_tz,\delta D_z)$ admit the eigenvalue $\delta \sigma$ with multiplicity~$1$. Therefore,  $\widehat P_t$ has a non zero vector in its kernel if and only if we choose $\rho_t=-\delta \sigma$, and  this choice allows us to construct an edge-mode. 
\item[(ii)]
Regarding the model operator $\Gamma_{\alpha,\gamma}(\sigma,r_tz,\delta D_z)$, a choice of $\rho_t$ is possible if and only if $\gamma=0$ (which implies $\alpha=\frac 1{\sqrt 2}$) and $\delta=-1$. Moreover, both  $\sigma$ and $-\sigma$ are the only simple (non zero) eigenvalues of $ \Gamma_{\frac 1{\sqrt2},0}(\sigma,r_tz,- D_z)$ that are linear in $\sigma$. Therefore, it is possible to construct two edge modes, corresponding to the choice of $\rho_t=+1$ or $\rho_t=-1$, which implies that these eigenmodes move in different directions on the curve~$\mathcal E$. 
\end{enumerate}

\begin{remark} \label{rem:wigner}Due to the existence of a Hilbertian basis of eigenvectors for the model problems (see Proposition \ref{prop:spectral}), it is possible to implement the strategy implemented for 2-band models in~\cite{vacelet}. This method gives a description of  the evolution of semiclassical measures of families $(\psi^h(t))_{h>0}$ of solutions to equation~\eqref{eq:dirac}. The advantage of this approach is that it does not require many assumptions  about the initial data in \eqref{eq:dirac}: the data are only supposed to be uniformly bounded in $L^2(\R^2,\C^3)$ with respect to the semiclassical parameter. We do not develop this point here. 
\end{remark}

Finally, in order to achieve the proof of Theorem~\ref{theo:main}, one has to be able to relate the original Hamiltonian~\eqref{def:H} to the model one $\Gamma_\bullet$. This is done by relating the properties of the matrices $\Theta_0$, $\Theta_1$ and $\Theta_2$ in~\eqref{def:H} to those appearing in $\Gamma_\bullet$ (see Section~\ref{sec:proof}). This allows us to identify whether we are in the cases~(i) or (ii) just described above. 
Note also that only the matrix $\Gamma_{\rm JT}(\sigma,z,\zeta)$ has the specificity of having a constant vector in its kernel, the relation~\eqref{eq:reductio} implying that the original Hamiltonian must have the same property. 
\smallskip

\subsection{Organisation of the paper}
The different steps of the proof of Theorem~\ref{theo:main} are described in the  following two sections. In Section~\ref{sec:reduc}, we first prove Theorem~\ref{theo:model_pb} and then study the model problems in a second subsection. Finally, we prove our main result in Section~\ref{sec:proof}, starting by establishing a condition of existence  of edge states. Then, we reduce to a model problem, which allows us to characterize the settings where edge-state exist. Finally, we find a criterion
for relating a given Hamiltonian to a model problem.   
\smallskip

\textsc{Acknowledgments.} The authors thank Éric Vacelet for
fruitful discussions. The authors acknowledge the support
of the Région Pays de la Loire via the Connect Talent Project HiFrAn 2022
07750, the grants ANR-23-CE40-0016 (OpART) and ANR-25-CE40-7296 (La Gabare), and the France 2030 program, Centre Henri Lebesgue
ANR-11-LABX-0020-01.
%%%%%%%%%%%%%%%%%%%%%%%%%%%%%%%%%%%%%%%%%%%%%%%%
%%%%%%%%%%%%%%%%%%%%%%%%%%%%%%%%%%%%%%%%%%%%%%%%%

%%%%%%%%%%%%%%%%%%%%%%%%%%%%%%%%%%%%%%%%%%%%%%%
%%%%%%%%%%%%%%%%%%%%%%%%%%%%%%%%%%%%%%%%%%%%%%

\section{Derivation and analysis of the  model problems}\label{sec:reduc}

In this section, we first prove Theorem~\ref{theo:model_pb}. Then, we analyze the spectrum of each of the model problems exhibited in Theorem~\ref{theo:model_pb}.

\subsection{Reduction to model problems}

For proving Theorem~\ref{theo:model_pb}, we will use the unitary matrices defined for $\theta\in\R$ by 
\[
D_{1,\theta}={\rm Diag} (\e^{i\theta},1,1),\;\;
D_{2,\theta}={\rm Diag} (1,\e^{i\theta},1),\;\;
D_{3,\theta}={\rm Diag} (1,1,\e^{i\theta}).
\]
Then, if
$\displaystyle{
M(\rho,\alpha,\beta,\gamma)=
\begin{pmatrix}
\rho  & \gamma & \alpha \\
\bar \gamma & -\rho & \beta \\
\bar \alpha & \bar \beta & 0 
\end{pmatrix}},
$
we have 
\begin{align}\label{lem:D1theta}
&D_{1,\theta}  
M(\rho,\alpha,\beta,\gamma)
D_{1,\theta}^*= M(\rho,\e^{i\theta}\alpha,\beta,\e^{i\theta}\gamma)
,\\
\label{lem:D2theta}
&D_{2,\theta}  
M(\rho,\alpha,\beta,\gamma)
D_{2,\theta}^*= 
M(\rho,\alpha,\e^{i\theta}\beta,\e^{-i\theta}\gamma),\\
\label{lem:D3theta}
&D_{3,\theta} 
M(\rho,\alpha,\beta,\gamma)
D_{3,\theta}^*= 
M(\rho,\e^{-i\theta}\alpha,\e^{-i\theta}\beta,\gamma).
\end{align}

The proof of Theorem~\ref{theo:model_pb} consists of several steps. We first derive properties of the coefficients of the matrix $\Gamma(\sigma,z,\zeta)$. Then we analyze them successively. 
\medskip 

\noindent{\bf 1st step.} We prove the following result.

\begin{lemma}\label{lem:coefficient}
Let $\Gamma\in\mathcal L(\R^3, \C^{3\times 3})$ be a linear map valued in the set of self-adjoint matrices with eigenvalues 
\[
-\sqrt{\sigma^2+z^2+\zeta^2}, \, 0,\, \sqrt{\sigma^2+z^2+\zeta^2}.
\]
Then, there exist a unitary matrix $U_1$ and three linear forms from $\R^2$ to $\C$, denoted by $\alpha$, $\beta$ and $\gamma$  such that for all $u=(\sigma,z,\zeta)\in\R^3$, 
    \begin{align}\label{eq:conjugaison}
    & U_1\Gamma(\sigma,z,\zeta) U^*_1=\begin{pmatrix}
    \sigma & \gamma(z,\zeta) & \alpha (z,\zeta)\\
    \overline \gamma(z,\zeta) & -\sigma  &\beta(z,\zeta)\\
    \overline \alpha (z,\zeta) & \overline \beta (z,\zeta) & 0
    \end{pmatrix},\\
    \label{eq:conditions}
       |\alpha(z,\zeta)|^2 = |\beta(z,\zeta)|^2,\;\;&
      |\alpha(z,\zeta)|^2+|\beta(z,\zeta)|^2+|\gamma(z,\zeta)|^2= z^2+\zeta^2,\;\; {\rm Re}(\overline \alpha(z,\zeta)  \beta(z,\zeta) \gamma(z,\zeta))=0. 
    \end{align}
    Moreover, we can assume $\alpha(z,\zeta)=\alpha_1 z+\alpha_2 \zeta$ with $\alpha_1\in \R_+\cup\{0\}$.
\end{lemma}

\begin{proof} We consider the matrix $\Gamma(\sigma, 0,0)$.  The set of its eigenvalues is $\{0,-|\sigma|,|\sigma|\}=\{0,-\sigma,\sigma\}$, and $\Gamma$ writes 
\[
\Gamma(\sigma,0,0)= \sigma \Gamma_0
\]
with $\Gamma_0$ a self-adjoint matrix with eigenvalues  $0,-1,1$. Therefore, there exists a unitary matrix $U_1$ such that 
\[U_1\Gamma_0 \,U_1^*={\rm Diag} (1,-1,0).
\]
We deduce the existence of a linear map $(z,\zeta)\mapsto W(z,\zeta)$ in the set of self-adjoint matrices such that 
\[
U_1\Gamma(\sigma,z,\zeta)U_1^* =\sigma\, {\rm Diag} (1,-1,0) + W(z,\zeta).
\]
Moreover ${\rm Tr}\, W(z,\zeta)=0$, which implies that $W$ depends on five linear maps $\alpha,\beta,\gamma, \rho, \mu $
such that 
\[
W(z,\zeta)=\begin{pmatrix}
\rho (z,\zeta) & \gamma(z,\zeta) & \alpha(z,\zeta) \\
\overline \gamma(z,\zeta) & \mu(z,\zeta) & \beta(z,\zeta)\\
\overline \alpha(z,\zeta) & \overline \beta(z,\zeta) & -(\rho+\mu)(z,\zeta)
\end{pmatrix}
\]
These five linear forms must satisfy some conditions so that the spectrum of $U_1 \Gamma(\sigma,z,\zeta)U_1^*
 $ has the expected properties. The polynomial function of four variables 
 \[
 P(\lambda, \sigma,z,\zeta)={\rm det} \left(\sigma\, {\rm Diag} (1,-1,0) + W(z,\zeta)-\lambda {\rm Id}\right) 
 \]
has to satisfy 
\[
 P(\lambda, \sigma,z,\zeta)=-\lambda (\lambda^2-\sigma^2-z^2-\zeta^2),\;\;(\lambda,\sigma,z,\zeta)\in\R^4.
\]
 A straightforward calculus gives 
\begin{align*}
 P(\lambda, \sigma,z,\zeta)
 =&-\lambda^3+\lambda(\sigma^2+\sigma(\rho-\mu)+\mu\rho+\rho^2+\mu^2+|\alpha|^2+|\beta|^2+|\gamma|^2)\\
 &+\sigma^2(\rho+\mu)+\sigma(|\alpha|^2-|\beta|^2+\rho^2-\mu^2)\\
 &-\mu\rho^2-\mu^2\rho+\rho(|\gamma|^2-|\beta|^2)+\mu(|\gamma|^2-|\alpha|^2)+\bar{\alpha}\beta\gamma+\alpha\bar{\beta}\bar{\gamma}.
\end{align*}
Identifying the 
coefficients of $\lambda$, and the term of order zero, which does not depend on $\lambda$, we obtain 
\begin{equation*}
\begin{cases}
&z^2+\zeta^2=\sigma(\rho-\mu)+\mu\rho+\rho^2+\mu^2+|\alpha|^2+|\beta|^2+|\gamma|^2\\
&0=\sigma^2(\rho+\mu)+\sigma(|\alpha|^2-|\beta|^2+\rho^2-\mu^2)-\mu\rho^2-\mu^2\rho\\
 &\;\;\;\;\;\;+\rho(|\gamma|^2-|\beta|^2)+\mu(|\gamma|^2-|\alpha|^2)+\bar{\alpha}\beta\gamma+\alpha\bar{\beta}\bar{\gamma}
\end{cases}
\end{equation*}
Viewed as a polynomial in $\sigma$,
we obtain the conditions 
\begin{align*}
& \rho +\mu=0,\;\; \rho-\mu=0,\;\;\mbox{whence}\;\;\rho=\mu=0,\\
&|\alpha|^2-|\beta|^2 =0,\\ 
& (\bar \alpha \beta\gamma +\alpha \bar \beta \bar \gamma)% -\rho |\alpha|^2 -\mu |\beta|^2
=0,\\
&|\alpha|^2+|\beta|^2+|\gamma|^2= z^2+\zeta^2.
\end{align*}
It remains to prove that we can choose $\alpha(z,\zeta)=\alpha_1 z +\alpha_2 \zeta$ with $\alpha_1\in \R_+$. If $\alpha_1=0$, we have nothing to do. If not, we write $\alpha_1=|\alpha_1|\e^{i\theta}$ and turn $U_1$ into $D_{3,\theta} U_1$. By~\eqref{lem:D3theta}, it replaces $\alpha_1 $ by~$|\alpha_1|$.
\end{proof}

\medskip

\noindent{\bf 2nd step. Study of $\alpha$ and $\beta$.}
We start with the
 special case where $\alpha=0$. Then $\beta=0$ and we deduce from the relation
\[
|\gamma(z,\zeta)|^2= z^2+\zeta^2,\;\;\forall z,\zeta\in\R
,
\]
that
$\gamma(z,\zeta)=\gamma_1z+\gamma_2\zeta$ with
$|\gamma_1|^2=|\gamma_2|^2=1$ and ${\rm Re}(\gamma_1\overline{\gamma_2})=0$. We obtain $\gamma_1=\e^{i\lambda}$ for some $\lambda\in \R$ and $\gamma_2=\pm i \gamma_1$. By~\eqref{lem:D1theta}, the conjugation by the matrix $D_{1,-\lambda}$ reduces us to the case $\lambda=0$. We are left with  case $\Gamma_{\rm JT}$ in the Theorem.
\smallskip 

We now assume that the linear form $\alpha$ is non identically zero and  we  focus on the relation~\eqref{eq:conditions}.

\begin{lemma}
Let $\alpha,\beta$ be as in Lemma~\ref{lem:coefficient} with $\alpha$ non identically zero.
Then, there exists $\lambda\in\R$ such that 
\[
\beta(z,\zeta)=\e^{i\lambda} \alpha(z,\zeta)\;\;\mbox{or}\;\; \beta(z,\zeta)=\e^{i\lambda} \overline\alpha(z,\zeta).
\]
Moreover, if $\alpha_1=0$, one can reduce to the case  $\beta(z,\zeta)=e^{i\lambda}\alpha(z,\zeta)$.
\end{lemma}

\begin{proof}
Let us write 
\[
\forall z,\zeta\in\R,\;\;\beta(z,\zeta)=\beta_1z+\beta_2 \zeta,
\]
for some $\beta_1,\beta_2\in\C$. 
The polynomial relation $|\alpha(z,\zeta)|^2= |\beta(z,\zeta)|^2 $ implies 
\[
\alpha_1=|\beta_1|,\;\; |\alpha_2|=|\beta_2|,\;\;\alpha_1{\rm Re} (\bar \alpha_2)={\rm Re} (\beta_1\bar \beta_2).
\]
If $\alpha_1=0$, necessarily  $\alpha_2\not=0$, $\beta_1=0$ and $\beta_2=\e^{i\lambda} \alpha_2$, the conclusion follows. \\
If $\alpha_1\not=0$, we consider $\widehat{(\alpha_1,\alpha_2)}$, the angle between the vectors of the complex plan associated with $\alpha_1$ and $\alpha_2$. Let $\widehat {(\beta_1,\beta_2)}$ be defined in a similar manner. We have 
$\widehat {(\beta_1,\beta_2)}= \pm \widehat {(\alpha_1,\alpha_2)}$.
We deduce the existence of $\lambda\in \R$ such that 
one of the two following facts holds: (i)
$\beta_1= \e^{i\lambda}\alpha_1$ and $\beta_2 =\e^{i\lambda} \alpha_2$,
or (ii)
$\beta_1= \e^{i\lambda}\alpha_1$ and $\beta_2 =-\e^{i\lambda} \alpha_2$.
We deduce $\beta(z,\zeta)=\e^{i\lambda} \alpha(z,\zeta)$ or $\beta(z,\zeta)=\e^{i\lambda} \overline\alpha(z,\zeta)$. 
\end{proof}

\medskip

\noindent{\bf 3rd step. Study of $\gamma$.}
At this stage of the proof, using the matrix $D_{2,-\lambda}$ and turning the matrix $U_1$ in~\eqref{eq:conjugaison} into $D_{2,-\lambda}$, equation~\eqref{lem:D2theta} ensures that we can assume 
\[
\alpha(z,\zeta)=\alpha_1 z+\alpha_2 \zeta\;\;\mbox{ with}\;\; \alpha_1\in\R_+
\]
with the alternative 
\begin{align*}
\mbox{if}\;\; \alpha_1=0,\;\; & 
\beta(z,\zeta)= \alpha_2 \zeta,\\
\mbox{if}\;\; \alpha_1>0,\;\; & 
\beta(z,\zeta)= \alpha_1 z + \alpha_2 \zeta\;\;\mbox{or}\;\;
\beta(z,\zeta)= \alpha_1 z + \overline {\alpha_2} \zeta.
\end{align*}
We now aim to identify the linear form $\gamma$ such that the two last equations of~\eqref{eq:conditions} hold. We have two cases to analyze since either $\beta=\alpha$ or $\beta=\bar\alpha$. 

\medskip

(i) {\bf Case 1.} Assume  first $\beta=\alpha\not=0$. Then, the linear form $\gamma$ satisfies~\eqref{eq:conditions} if and only if 
\[
|\gamma(z,\zeta)|^2=z^2+\zeta^2-2 |\alpha(z,\zeta)|^2\;\;\mbox{and}\;\; {\rm Re}(\gamma(z,\zeta))=0.
\]
We deduce ${\rm Re}(\gamma)=0$. Therefore, the linear form $\gamma$ writes $\gamma(z,\zeta)=i(\gamma_1 z+\gamma_2\zeta)$ with $\gamma_1,\gamma_2\in \R$. In view of the fact that
\begin{equation}\label{eq:module}
|\gamma(z,\zeta)|^2 = z^2+\zeta^2 - 2 |\alpha(z,\zeta)|^2,
\end{equation}
we have
\[
\gamma_1^2=1-2\alpha_1^2,\;\; \gamma_2^2=1-2|\alpha_2|^2,\;\;\gamma_1\gamma_2 =-2\alpha_1 {\rm Re}(\alpha_2).
\]
Therefore, we look for $\alpha_1,\alpha_2,\gamma_1,\gamma_2$ such that 
    \begin{align}
   \label{cond:alpha1}
  & (\gamma_1,\gamma_2)= \Big(\eps \sqrt{1-2\alpha_1^2} , \eps' \sqrt{1-2|\alpha_2|^2}\Big),\;\;\eps,\eps'\in\{1,-1\},\;\;
\alpha_1,|\alpha_2|\in\Big[0,\frac 1{\sqrt 2}\Big].  
\end{align}

The alternative is the following:

$\bullet$ If $\gamma=0$, then $\alpha_1=|\alpha_2|= \frac 1{\sqrt 2}$, ${\rm Re}(\alpha_2)=0$. We are left with a situation that belongs to  case~$\Gamma_{\rm PJT}$.

$\bullet$ If $\gamma\not=0$, we set abusively $|\gamma|=\sqrt{\gamma_1^2+\gamma_2^2}$ and  perform
the change of symplectic coordinates $(z,\zeta)\mapsto (z',\zeta')$ such that 
\[ z'= \frac{1}{|\gamma|}
(\gamma_1z+\gamma_2 \zeta),\;\; 
\zeta'= \frac{1}{|\gamma|}
(-\gamma_2z+\gamma_1 \zeta).\]
This change of coordinates also preserves the norm: $z^2 +\zeta^2=(z')^2+(\zeta')^2$. Set 
\[
\alpha'(z',\zeta')=\alpha'_1z'+\alpha'_2\zeta'= \alpha(z,\zeta),
\]
the relation~\eqref{eq:module} yields for all $(z',\zeta')\in\R^2$,
\[
|\gamma|^2 (z')^2 = (z')^2+(\zeta')^2 -2 |\alpha'(z',\zeta')|^2= (1-2|\alpha'_1|^2)(z')^2 +(1-2|\alpha'_2|^2)(\zeta')^2 -4{\rm Re}(\alpha_1'\overline{\alpha_2'}) z\zeta.
\]
We deduce that there exists $\lambda\in\R$ such that 
\[\alpha'_2= \frac 1{\sqrt 2}{\rm e}^{i\lambda}, \;\;|\gamma|= \sqrt{1-2|\alpha'_1|^2},\;\;\alpha_1'= \pm i |\alpha_1|{\rm e}^{i\lambda}=|\alpha_1|{\rm e}^{i(\lambda +\frac \pi 2 \pm \pi)}.
\]
Using the matrix $D_{3,\lambda+\frac\pi 2 \pm \pi}$ as in~\eqref{lem:D2theta} allows us to reduce to a matrix that belongs to case~$\Gamma_{\alpha,\gamma}$.  
\medskip 

 (ii) {\bf Case 2.} Assume that $\beta=\bar \alpha\not=\alpha $ with $\alpha_1\not=0$. We are going to prove that we are in case $\Gamma_{\rm PJT}$. 
The fact that $\beta=\bar \alpha\not=\alpha$ implies $\Im(\alpha_2)\not=0$.
We first analyze the relation
\begin{equation}\label{eq:triple}
{\rm Re} (\bar\alpha(z,\zeta)^2 \gamma(z,\zeta))=0,\;\; z,\zeta\in\R.
\end{equation}
Taking $(z,\zeta)=(1,0)$ and then $(z,\zeta)=(0,1)$, we obtain  
\[
\alpha_1^2{\rm Re} (\gamma(1,0))=0\;\;\mbox{and}\;\; {\rm Re} (\bar \alpha_2^2\gamma(0,1))=0.\]
Regarding condition ${\rm Re} (\bar \alpha_2^2\gamma(0,1))=0$, 
since $\alpha_2\not=0$, we can write 
 $\alpha_2=|\alpha_2|\e^{i\varphi}$ with $\sin\varphi\not=0$, and we deduce the existence of $\gamma_1,\gamma_2\in\R$ such that 
\[
\gamma(z,\zeta)=i\gamma_1 z + i \gamma_2\e^{2i\varphi}\zeta,\;\;z,\zeta\in\R.
\]
Moreover, the relation $2|\alpha(z,\zeta)|^2+|\gamma(z,\zeta)|^2=z^2+\zeta^2$ (see~\eqref{eq:conditions}) implies that
%and~\eqref{eq:triple}, 
$\gamma_1$ and $\gamma_2$ satisfy the equation
\[
\gamma_1^2=1-2\alpha_1^2,\;\; \gamma_2^2=1-2|\alpha_2|^2,\;\; 0<\alpha_1\leq \frac{1}{\sqrt 2},\;\;|\alpha_2|\leq\frac1{\sqrt 2},\;\;
\gamma_1 |\gamma_2|\cos(2\varphi)= -2\alpha_1 |\alpha_2| \cos \varphi. 
\]
and the relation ${\rm Re}(\bar \alpha\beta\gamma)=0$ (see~\eqref{eq:triple})  yields
\begin{align}
\label{fifi}
0=&{\rm Re}(i\alpha_1^2\gamma_2\e^{2i\varphi}+2i\alpha_1\bar\alpha_2\gamma_1)
=  2\alpha_1\sin\varphi(-\alpha_1\gamma_2\cos\varphi+|\alpha_2|\gamma_1),\\
\label{riri}0= & {\rm Re} (i\bar\alpha^2_3\gamma_1+2i\alpha_1\bar\alpha_2\gamma_2\e^{2i\varphi})
= 2|\alpha_2|\sin\varphi (\gamma_1|\alpha_2| \cos\varphi -\gamma_2\alpha_1).
\end{align}
The latter two conditions come from the identification of  the coefficients of $z^2 \zeta$ and $\zeta^2 z$ in~\eqref{eq:triple}. We are left with three cases:
\begin{enumerate}
\item If $\gamma_1=0$, i.e. $\alpha_1=\frac 1{\sqrt 2}$, we first deduce $\cos\varphi=0$ and $\varphi=\pm\frac \pi 2$. Equation~\eqref{fifi} is satisfied and~\eqref{riri} yields $\gamma_2\alpha_2=0$, whence $\gamma_2=0$ since we have  $\alpha_2\not=0$. We obtain the solution 
\[
\alpha_1=\frac 1{\sqrt2},\;\alpha_2=-\delta \frac 1{\sqrt 2}, \;\text{with } \delta=\pm 1, \;\gamma=0.
\]
\item If $\gamma_2=0$ and $\gamma_1\not=0$, then $|\alpha_2|=\pm  \frac 1{\sqrt 2}$, $\cos\varphi=0$, equation~\eqref{riri} is satisfied and \eqref{fifi} yields $\gamma_1\alpha_1=0$, which is excluded.
\item  If $\gamma_2\gamma_1\not=0$, then $0<\alpha_1<1/\sqrt 2$ and $0\leq |\alpha_2|<1/\sqrt2$. Moreover, in view of the equation $\gamma_1 |\gamma_2|\cos(2\varphi)= -2\alpha_1 |\alpha_2| \cos \varphi$, we deduce that the polynomial function 
\[
P(X)= 2\gamma_1|\gamma_2|X^2+2 X\alpha_1|\alpha_2|  -\gamma_1|\gamma_2|
\]
vanishes in $[-1,1]$. 
This  requires $P(1)P(-1)\leq  0$, i.e.
$\gamma_1^2 \gamma_2^2 \leq  4 \alpha_1^2|\alpha_2|^2$,
whence
$\alpha_1^2 +|\alpha_2|^2\geq  \frac 1 2$. This condition excludes the case $\alpha_2=0$ because it would imply   $\alpha_1^2=1/2$ and $\gamma_1=0$. Therefore, necessarily $\alpha_1,\alpha_2\in(0,1/\sqrt 2)$ and $P(1)P(-1)\not=0$, which implies $\cos\,\varphi\notin\{+1,-1\}$ and $\sin\,\varphi\not=0$.  
By the  relations~\eqref{fifi} and~\eqref{riri}, we obtain 
\[\cos\varphi=
\frac{|\alpha_2|\gamma_1}{\alpha_1\gamma_2}=\frac{\alpha_1\gamma_2}{|\alpha_2|\gamma_1},
\]
whence $\frac{\alpha_1^2}{\gamma_1^2}=\frac{|\alpha_2|^2}{\gamma_2^2}$, which implies $\alpha_1^2=|\alpha_2|^2$ (using $\gamma_1^2=1-2\alpha_1^2$ and $\gamma_2^2= 1-2|\gamma_2|^2$), and $\gamma_1^2=\gamma_2^ 2$, thus  $\cos\varphi=\pm 1$, which was excluded. 
\end{enumerate}
As a conclusion, only case (1) may happen. This concludes the proof of Theorem~\ref{theo:model_pb}.

\subsection{The spectrum of the model operators}
Let $r>0$.
We  analyze the operators $\Gamma_\bullet(\sigma, r z,\delta D_z)$ where the index ``$\bullet$'' stands for JT, PJT or $\alpha,\gamma$ with $1=\gamma^2+2\alpha^2$. The spectrum of the operator $\Gamma_{\rm JT}(\sigma, r z, \delta D_z)$ has been studied in~\cite{vacelet}. Let $(h_{n,r})_{n\in\N}$ be the Hilbertian basis of eigenfunctions  of the  operator $r^2z^2-\partial_z^2$. We have 
\[ 
h_{n,r}(z)=r^{\frac 14} h_n(z\sqrt r),\;\; z\in\R,
\]
where $(h_n)_{n_in\N}$ is the family of Hermite functions:
\[ h_n(z)=\frac{(-1)^n }{\sqrt{2^n n!\sqrt \pi}} {\rm e}^{\frac {z^2}2}\left(\frac d{dz}\right)^n{\rm e}^{-z^2 },\;\; z\in\R.\]
In that manner, for $n\in\N$,
\begin{align*}
&(r^2z^2-\partial_z^2)h_{n,r}=r(2n+1)h_{n,r},\\
&\frac 1{\sqrt 2} (rz-\partial_z)h_{n,r}=\sqrt {r(n+1)}h_{n+1,r},\\ 
&\frac 1{\sqrt 2} (rz+\partial_z)h_{n,r}=\sqrt {rn}\, h_{n-1,r}, \; n\not=0.
\end{align*}
\smallskip 

\begin{proposition} \label{prop:spectral}
 (1) For all $\sigma\in\R$, we have 
\begin{align}\label{sp_H_2}
&{\rm Sp} (\Gamma_{\rm JT}(\sigma,rz,\delta D_z))=\{0\}\cup\{\delta\sigma\}\cup \left\{\lambda_{n,\pm}=\pm\sqrt{r(2n+1+\delta)+\sigma^2},\;\;n\in\N\setminus\{0\}\right\}.
\end{align}
The eigenvalue $0$ has infinite multiplicity  and the other eigenvalues are simple. 

(2) Denote for $n\in\N$ and $r>0$ by
$\mu_{n,j}(\sigma,r)$, $1\leq j\leq 3$,  the three real-valued solutions of the equation 
\[
-\mu^3+\mu(\sigma^2+(2n+1)r)+\sigma \delta r=0,
\]
Then, for all $\sigma\in\R$, 
\begin{align}
\label{sp_H_3(a)}
{\rm Sp} (\Gamma_{\rm PJT}(\sigma,rz,\delta D_z))=\{\delta\sigma\}\cup \{-\delta\sigma \pm\sqrt{\sigma^2+4r}\}\cup\left\{\mu_{n,j}(\sigma,r),\;\;n\in\N\setminus\{0\},\;\;1\leq j\leq 3\right\},
\end{align}
and all these eigenvalues have multiplicity $1$. 

(3) If $\alpha\in[0,1/\sqrt 2)$ or ($\alpha=1/\sqrt 2$ with $\delta=1$),
then, for $\gamma\in\R$ such that $\gamma^2=1-2\alpha^2$, and for all $\sigma\in\R$, 
\begin{align}
&{\rm Sp} ( \Gamma_{\alpha,\gamma}(\sigma,rz,\delta D_z))=\{0\}\cup \left\{\lambda_{n,\pm}(\sigma,r)=\pm \sqrt{\sigma^2 +r(2n+1+\sqrt 2 \delta \,\alpha)},\;\;n\in\N\right\}.
\end{align}
%\footnote{régler s'il y a $0$ ou pas}
If $\alpha=1/\sqrt 2$, $\gamma=0$ and $\delta=-1$, then 
for all $\sigma\in\R$, 
\begin{align}
&{\rm Sp} ( \Gamma_{1/\sqrt2 ,0}(\sigma,rz,-D_z))=\{0\}\cup \{\sigma,-\sigma\}\cup \left\{\lambda_{n,\pm}(\sigma,r)=\pm \sqrt{\sigma^2 +r(2n+1+\sqrt 2 \delta \,\alpha)},\;\;n\in\N\right\}.
\end{align}
Moreover, in each of the previous three cases, there exists a Hilbertian basis of eigenfunctions of the operator $\Gamma_\bullet(\sigma,rz,\delta D_z)$, and the eigenvectors associated with the eigenvalues $0$, $\sigma$ or $-\sigma$, when they are eigenvalues, can be chosen independently of $\sigma$. 
\end{proposition}

\begin{remark}\label{rem:kernel}
In contrast with the case of $2\times 2$ matrices in~\cite{BBDFLW}, we have different situations: either there exists  no (non zero) eigenvalue that is a linear function of $\sigma$, either there are two, or only one, as in~\cite{BBDFLW}.
 Another difference from the case of $2\times 2$ matrices is the presence of $0$ in the spectrum.
\end{remark}

\begin{proof}
 (1) The spectrum of the operator $\Gamma_{\rm JT}(\sigma, r z, D_z)$ has been studied in~\cite[Appendix~B]{vacelet}. For $\delta=1$, after normalisation, the vectors 
 \[
 \begin{pmatrix} 
h_{0,r}\\ 0 \\ 0
 \end{pmatrix},\;\; \begin{pmatrix}
\frac{r\sqrt {2n}}{\lambda_{n,\pm}-\sigma}h_{n,r}\\
h_{n-1,r}\\ 0
 \end{pmatrix},\;\;
 \begin{pmatrix}0\\0\\ h_{m,r}\end{pmatrix},\;\; n\in \N^*,\;\;m\in\N, \]
 give a Hilbertian basis of $L^2(\R,\C^3)$. A similar result holds when $\delta=-1$.
 \smallskip 
 
\noindent (2) {\it Spectrum of $\Gamma_{\rm PJT}(\sigma, rz, D_z)$.}
Let $F=\,^t(f,g,h)\in (L^2(\R))^3\setminus\{\,^t(0,0,0)\}$  and $\lambda\in\R$ be such that 
\[
 \Gamma_{\rm PJT}( \sigma, r z, \delta D_z)F=\lambda F.
\]
Set $\mathfrak b=\frac r{\sqrt 2} z-\frac\delta{\sqrt2} \partial_z$. Note that $\mathfrak b$ depends on $\delta$: it  is a creation operator when $\delta=1$ and an anihilation operator if $\delta=-1$ (with the usual terminology). 
We have 
\begin{align*}
&(\sigma-\lambda) f  + \mathfrak b h=0,\\
&  (-\sigma-\lambda) g +  \mathfrak b^* h=0,\\
&\mathfrak b^*f+ \mathfrak bg -\lambda h=0.
\end{align*}

Let us first determine the eigenvalues that satisfy  $\lambda^2\not=\sigma^2$. We multiply the last equation by $\lambda^2-\sigma^2$, and inject the value of $(\sigma-\lambda)f$ and $(-\sigma-\lambda)g$ given by the first two. We obtain
\[
(-\lambda^3 +\lambda\sigma^2 + \sigma(\mathfrak b^*\mathfrak b-\mathfrak b\mathfrak b^*)+ \lambda (\mathfrak b^*\mathfrak b+\mathfrak b\mathfrak b^* ))h=0.
\]
In view of 
\[ \mathfrak b^*\mathfrak b+\mathfrak b\mathfrak b^*=r^2z^2-\partial_z^2 \;\;\mbox{and}\;\;\mathfrak b^*\mathfrak b-\mathfrak b\mathfrak b^*= \delta  r , \]
we deduce that $h$ and $\lambda$ must satisfy 
\[
(-\lambda^3+ \lambda\sigma^2 +\lambda(r^2z^2-\partial_z^2)  +\sigma  \delta  r)h=0.
\]
Therefore, for $h\not=0$, we have one of the relations 
\[
p_n(\lambda)=-\lambda^3 +\lambda  (\sigma^2 +(2n+1)r)   +\sigma  \delta  r=0,\;\: n\in\N,
\]
with $h\in{\rm Span}\, h_{n,r}$.
Indeed, $h=0$ implies $f=g=0$, which contradicts $F\not=0$.  
We observe that
the polynomial function $p_n(\lambda)=-\lambda^3+\lambda(\sigma^2+(2n+1)r)+\sigma \delta r$ has two local extrema which are reached at  $-\frac 1{\sqrt 3}\sqrt{\sigma^2+(2n+1)r}$ and $\frac 1{\sqrt 3}\sqrt{\sigma^2+(2n+1)r}$, with
\[
p_n\left(\frac 1{\sqrt 3}\sqrt{\sigma^2+(2n+1)r}\right)p_n\left(-\frac 1{\sqrt 3}\sqrt{\sigma^2+(2n+1)r}\right)
= \sigma^2 r^2 -\frac 4{27}(\sigma^2+(2n+1)r)^3
.
\]

(i) When $n\not=0$, $\sigma^2 r^2 -\frac 4{27}(\sigma^2+(2n+1)r)^3\leq \sigma^2r^2-\frac 4{27} (\sigma^2+3r)^3<0$.
This ensures the existence of three distinct solutions  $\mu_{n,j}(\sigma,r)$, $1\leq j\leq 3$,  to equation $p_n(\lambda)=0$.
Moreover, $\mu_{n,j}^2\not=\sigma^2$, because 
$p_n(\pm\sigma)=
\pm\sigma r(2n+1\mp\delta)\not=0$ for $n\not=0$.
We assume that the indices have been chosen so that  $\mu_{n,1}< \mu_{n,2}<\mu_{n,3} $. With each of these eigenvalues is associated  a eigenspace of dimension~1 spanned by the vector $F_{n,j}^\delta =^t(f_{n,j}^{\delta},g_{n,j}^\delta ,h_{n,j}^{\delta})$ with 
\begin{equation}\label{eigenvectorFnj}
f_{n,j}^{\delta}=-(\sigma-\mu_{n,j})^{-1}\mathfrak bh_{n,r},\;\;g_{n,j}^{\delta}= (\sigma+\mu_{n,j})^{-1}\mathfrak b^*h_{n,r},\;\; h_{n,j}^{\delta}=h_{n,r},\;\;n\not=0,\;\; 1\leq j\leq 3.
\end{equation}

(ii) When $n=0$, $p_0(\lambda)=(\lambda+\delta\sigma)(-\lambda^2+\lambda\sigma \delta +r)$,  and  $p_0$ has three roots, $-\delta\sigma-\sqrt{\sigma^2+4r}$, $-\delta\sigma$, and $-\delta\sigma+\sqrt{\sigma^2+4r}$. Only two of these three roots satisfy $\lambda^2\not=\sigma^2$. With the choice of indexation, they write $\mu_{0,1}$ and $\mu_{0,3}$. They are simple eigenvalues, with eigenvectors given by $F_{0,1}^\delta $ and $F_{0,3}^\delta $, respectively, according to the formula~\eqref{eigenvectorFnj}  which makes sense in this case. 
\smallskip 

Let us now examine whether there exist eigenvalues such that $\lambda^2=\sigma^2$. 
We start with $\lambda=\sigma$. Finding an eigenvector requires solving the equation $\mathfrak b h=0$. If $\delta=1$, this yields $h=0$, thus $g=0$ and $\mathfrak b^* f=0$. 
We find the eigenspace spanned by  $F_{0,2}^{1}:=(h_{0,r},0,0)$. If $\delta=-1$, one is rapidly convinced that it is not possible to solve the eigenvector system:  $h$ has to be taken colinear to $h_{0,r}$, whence $g$ colinear to $h_{1,r}
$ and $\mathfrak b^* f$ colinear to $h_{0,r}$, which is not possible, unless they are zero. 
Similarly, $\lambda=-\sigma$ is an eigenvalue only when $\delta=-1$. It is a simple eigenvalue, and an eigenvector is given by  $F_{0,2}^{-1}:= ^t(0,h_{0,r},0)$.
\smallskip 

In order to conclude the proof, one  checks that the vectors $F_{n,j}^\delta$, $n\in\N$, $1\leq j\leq 3$, form a Hilbertian basis. When $\delta=1$, this appears clearly when observing the eigenvectors:
considering $\tau_n$ such that $\mathfrak b h_{n,r}=\tau_n h_{n+1,r}$ for $n\geq 0$, and 
$\mathfrak b^* h_{n,r}=\tau_{n-1} h_{n-1,r}$ for $n\geq 1$, we have 
\begin{align*}
&F_{0,2}^1=\begin{pmatrix} h_{0,r} \\ 0 \\ 0 \end{pmatrix},\;\;
F_{0,1}^1=\begin{pmatrix} -(\sigma-\mu_{0,1})^{-1}h_{1,r} \\ 0 \\ h_{0,r} \end{pmatrix},\;\;
F_{0,3}^1=\begin{pmatrix} -(\sigma-\mu_{0,3})^{-1}h_{1,r} \\ 0 \\ h_{0,r} \end{pmatrix},\\
&F_{n,j}^1=\begin{pmatrix} -\tau_n(\sigma-\mu_{n,j})^{-1}h_{n+1,r} \\ \tau_{n-1}(\sigma+\mu_{n,j})^{-1}h_{n-1,r} \\ h_{n,r} \end{pmatrix},\;\;
j=1,2,3, \; n\in\N\setminus\{0\}.
\end{align*}
If a vector $\,^t(f,g,h)$  is orthogonal to all these vectors, then $(f,h_{0,r})= (h,h_{0,r})=(f,h_{0,r})=0$, It also gives an infinite series of  three by three homogeneous systems about $(f,h_{n+1,r})$, $(g,h_{n-1,r})$, $n\in\N$, and one can verify that for all $n\in\N$  the determinant of this system is non zero because the numbers $\mu_{n,1}$, $\mu_{n,2}$ and $\mu_{n,3}$ are three distinct numbers. 
The situation is similar when $\delta=-1$. 

\bigskip 
\noindent (3) {\it Spectrum of $\Gamma_{\alpha,\gamma}( \sigma, r z, D_z)$.} Let $G=\,^t(f,g,h)\in (L^2(\R))^3$ and $\lambda\in\R$ with 
\[
\Gamma_{\alpha,\gamma}( \sigma, r z, \delta D_z)G=\lambda G.
\]
We set $\mathfrak a=\frac1{\sqrt2}(r \alpha \sqrt2 z+\delta \partial_z)$.
Then, we have $\mathfrak a^*=\frac1{\sqrt2}(r \alpha \sqrt 2 z-\delta\partial_z)$ and 
\begin{align*}
&(\sigma-\lambda) f + i\gamma r z g + \mathfrak a h=0,\\
&-i\gamma r zf + (-\sigma-\lambda) g +  \mathfrak a h=0,\\
&\mathfrak a^*f+ \mathfrak a^*g -\lambda h=0.
\end{align*}
We rewrite these equations taking the sum and the difference of the two first equations and obtain
\begin{align*}
&\lambda h= \mathfrak a^*(f+g),\\
&\lambda(f-g)= (\sigma +i\gamma rz ) (f+g),\\
&-\lambda(f+g)+(\sigma -i\gamma rz)(f-g)  +2 \mathfrak ah =0
\end{align*}
If $\lambda\not=0$, multiplying the last equation by $\lambda$ and injecting the first two inside the third, we obtain 
\[
(-\lambda^2+  \sigma^2 +\gamma^2 r^2 z^2  +2\mathfrak a\mathfrak a^*)(f+g)=0.\]
We observe that 
\begin{align*}
\gamma^2 r^2 z^2  +2\mathfrak a\mathfrak a^*& = (\gamma^2 + 2 \alpha^2) r^2 z^2 -\partial_z^2 + \sqrt 2 \delta r\alpha  = r^2 z^2 -\partial_z^2 +\sqrt 2 \delta r\alpha,
\end{align*}
and we are left with the following equation:
\[
(-\lambda^2 +\sigma^2 +r^2 z^2 -\partial_z^2 +\sqrt 2 \delta r\alpha)(f+g)=0.
\]
 We find a first set of simple eigenvalues, 
 \[
 \lambda_{n,\pm}= \pm \sqrt{\sigma^2 +r(2n+1+\sqrt 2\delta \alpha)},\;n\not=0
 \]
 with the corresponding eigenvector 
$G_{n,\pm}^\delta= (f_{n,\pm}^\delta, g_{n,\pm}^\delta, f_{n,\pm}^\delta)$ given by 
\[
f_{n,\pm}+g_{n,\pm}=h_{n,r},\;\;f_{n,\pm}-g_{n,\pm}= \lambda_{n,\pm}^{-1} (\sigma +i\gamma rz) h_{n,r},\;\; 
h_{n,\pm}^\delta =\lambda_{n,\pm}^{-1} \mathfrak a^* h_{n,r}.
\]
Note that since $n\not=0$, we have $\lambda_{n,\pm}\not=0$.\\
The formula with  $n=0$ also gives a pair of eigenvalues and eigenvector  when  $\alpha\in (0,1/\sqrt 2)$ or $\alpha=1/\sqrt 2$ with $\delta=1$. We find the additional eigenvalues 
$\lambda_{0,\pm}=\pm \sqrt{\sigma^2 +r(1+\sqrt 2\delta \alpha)}$, with the eigenvectors $G_{0,\pm}^\delta$ obtained by extending the previous formula.\\
In the case where 
 $\alpha=1/\sqrt 2$ and $\delta=-1$, we  find the eigenvalue
$\lambda_{0,+}=\sigma$ for the eigenvector $G_{0,+}^\delta =(h_{0,r},0,0)$ and $\lambda=-\sigma$ for $G_{0,-}^\delta=(0,h_{0,r},0)$.
\smallskip

Consider now the case $\lambda=0$. We find $f=-g$ and $\mathfrak a h +(\sigma -i\gamma rz) f=0$. Fixing $f$ with enough regularity, one can find a corresponding $h$ by solving the equation 
\begin{equation}\label{eq:a*}
(2\alpha^2 r^2 z^2 -\partial_z^2-\delta \alpha r \sqrt 2)h=\sqrt 2(\alpha r\sqrt2 z -\delta  \partial_z ) (\sigma -i\gamma rz)f.
\end{equation}
The eigenvalues of the operator $2\alpha^2 r^2 z^2 -\partial_z^2-\delta \alpha r \sqrt 2$ are $\alpha r\sqrt 2(2n+1-\delta)$, $n\in\N$, with the Hilbertian basis of eigenvectors $(h_{n,\sqrt 2 \alpha r})_{n\in\N}$, so that  $h$ and $f$ satisfy 
\[
r\alpha\sqrt 2(2n+1-\delta) (h,h_{n,\alpha r\sqrt 2})_{L^2}=\sqrt 2
((\sigma-i\gamma rz)f, (\alpha r\sqrt 2 z-\delta\partial_z)h_{n,\alpha r\sqrt 2})_{L^2}. 
\]
In particular, when $\delta=1$, one has the solution $f=0$, $h=h_{0,\alpha r\sqrt 2}$.
We denote by $h=T(f)$ the function that satisfies 
\[
\mathfrak a T(f)+(\sigma-i\gamma r z)f=0,
\]
and which is a solution to~\eqref{eq:a*}, we choose $h=T(f)$.
Note that we  
have obtained an eigenspace of infinite multiplicity. 
\smallskip 

Let us now study the family of eigenvectors. All what we do below depends on $\sigma$. \\
(i) When $\delta=-1$ and $\alpha=\frac 1{\sqrt2}$, 
equation~\eqref{eq:a*} reduces to $(rz-\partial_z)h=-\sqrt 2\sigma f$.
Therefore, we have the set of eigenvectors 
\begin{align*}
\begin{pmatrix} 
h_{0,r}\\ 0\\0
\end{pmatrix},\;\;
\begin{pmatrix} 
0\\ h_{0,r}\\0
\end{pmatrix},\;\;
\begin{pmatrix} 
\frac 12 (1+\lambda^{-1}_{n,\pm} \sigma )h_{n,r}\\ 
\frac 12 (1-\lambda^{-1}_{n,\pm}\sigma )h_{n,r}\\
(1/\sqrt 2)\, \lambda^{-1}_{n,\pm} \tau_{n-1} h_{n-1,r}
\end{pmatrix},\;\; 
n\in\N\setminus\{0\},\;\;
\begin{pmatrix}
h_{p,r}\\ -h_{p,r}\\ -\frac{2\sigma}{\tau_p}h_{p+1,r}
\end{pmatrix}, \;\; p\in \N%\;\mbox{with} \; \mathfrak ah+(\sigma-i\gamma rz) f=0
\end{align*}
where we have set $\sqrt 2 \mathfrak a^*h_{n+1,r}= (rz+\partial_z) h_{n+1,r}=\tau_{n} h_{n,r}$. Using $\lambda_{n,+}=-\lambda_{n,-}$, we  can consider 
equivalently the family 
\begin{align*}
\begin{pmatrix} 
h_{0,r}\\ 0\\0
\end{pmatrix},\;\;
\begin{pmatrix} 
0\\ h_{0,r}\\0
\end{pmatrix},\;\;
\begin{pmatrix} 
h_{n,r}\\ 
h_{n,r}\\
0
\end{pmatrix},\;\; 
\begin{pmatrix} 
 \sigma h_{n,r}\\ 
-\sigma h_{n,r}\\
\sqrt 2\, \tau_{n-1} h_{n-1,r}
\end{pmatrix},\;\; 
n\in\N\setminus\{0\},\;\;
\begin{pmatrix}
h_{p,r}\\ -h_{p,r}\\ -\frac{2\sigma}{\tau_p}h_{p+1,r}
\end{pmatrix}, \;\; p\in \N%\;\mbox{with} \; \mathfrak ah+(\sigma-i\gamma rz) f=0
\end{align*}
which generates a Hermitian basis (after normalization). \\
(ii) When $\delta=-1$ and $\alpha\not=\frac 1{\sqrt2}$, linear combination of the eigenvectors, as those used in the previous case, suggests to look for a family of the form  
\begin{align}
\label{def:basisalpha}
U_n:=\begin{pmatrix} 
h_{n,r}\\ 
h_{n,r}\\
0
\end{pmatrix},\;\; 
V_n:=\begin{pmatrix} 
(\sigma +i\gamma z) h_{n,r}\\ 
-(\sigma+i\gamma z) h_{n,r}\\
 2\,  \mathfrak a^* h_{n,r}
\end{pmatrix},\;\; 
n\in\N,\;\;
W_\phi:=\begin{pmatrix}
\phi\\ -\phi\\ T(\phi)
\end{pmatrix}
\end{align}
for an adequate choice of vectors $\phi$ that we are going to perform now. The first point to notice is the orthogonality of these families. Indeed, by the definition of $T(\phi)$ we have 
\[
\left( V_n,W_\phi\right)_{L^2(\R,\C^3)}= 
2 (h_{n,r},(\sigma-i\gamma z)\phi)_{L^2(\R)} + 2(h_{n,r}, \mathfrak a T(\phi))_{L^2(\R)}=0.
\]
We define a Hermitian inner product $\Lambda$  on   $L^2(\R)$ by setting  
\[
\Lambda(u,v)= (T(u),T(v))_{L^2(\R)} + 2( u,v)_{L^2(\R)},\;\; u,v\in L^2(\R),
\]
and we construct $(\phi_n)_{n\in \N}$ by requiring 
\begin{align*}
&{\rm Span}\left(h_{k,r},\;\; 0\leq k\leq n\right)={\rm Span}\left(\phi_k,\;\; 0\leq k\leq n\right),\;\; n\in\N,\\
&\phi_{n+1}\in {\rm Span}\left(h_{k,r},\;\; 0\leq k\leq n\right)^\perp ,\;\; n\in\N,\\
&\Lambda(\phi_k,\phi_n)_{L^2(\R,\C^3)}=0,\;\; \forall k\not=n. 
\end{align*}
The third condition can be obtained via a Gram-Schmidt orthogonalization process  for the Hermitian inner product $\Lambda$.
With this choice, the family $\{U_n, V_n, W_{\phi_n},\,n\in\N\}$ is an orthogonal family of $L^2(\R,\C^3)$ and $(\phi_n)_{n\in\N}$  is a complete orthogonal family of $L^2(\R)$. To conclude the proof, it remains to show that if $\,^t(u,-u,w)$ is orthogonal to all the $V_n$-s and $W_{\phi_p}$-s, $n,p\in\N$, then, necessarily $u=w=0$. Such a pair $(u,w)$ verifies for all $n\in\N$,
\[ 
(u,(\sigma+i\gamma z) h_{n,r})_{L^2(\R)} + (w,\mathfrak a^* h_{n,r})_{L^2(\R)}=0 \;\;\mbox{and}\;\;
2(u,\phi_n)_{L^2(\R)} +(w, T(\phi_n))_{L^2(\R^n)}=0.
\]
We deduce from the first equation that $w=T(u)$, and the second gives  $\Lambda(u,\phi_n)=0$ for all $n\in\N$, whence  $u=0$.
We conclude that the family $\{U_n, V_n, W_{\phi_n},\,n\in\N\}$ is a Hermitian basis of $L^2(\R,\C^3)$, after normalisation. \\
(iii) When $\delta=1$, a similar analysis (that we let to the readers) can be performed. 
\end{proof}

%%%%%%%%%%%%%%%%%%%%%%%%%%%%%%%%%%%%
%%%%%%%%%%%%%%%%%%%%%%%%%%%%%%%%%%%%%%%%%%

\section{Proof of the main result: existence of edge states}\label{sec:proof}

In this section, we prove Theorem~\ref{theo:main}. We first explain why the existence of Schwartz functions in the kernel of the operator $\widehat P_t$ defined in~\eqref{def:Pt} is enough for the construction of edge-states. Then, using the reduction result of Theorem~\ref{theo:model_pb} and the spectral resolution of the model problems stated in Proposition~\ref{prop:spectral}, we are able to decide when some edge states exist. Finally, in the last subsection, we characterize each case by properties of the matrices $\Theta_0$, $\Theta_1$ and $\Theta_2$ in~\eqref{def:H}.  

\subsection{Existence of edge states}\label{sec:Edge}

The discussion of Section~\ref{subsec:strategy} has led us to the equation~\eqref{eq:edge_state}: the profile $\varphi^\hbar$ must satisfy
\begin{align*}
0= &\left(\dot x_t \cdot D_y+ y\cdot \nabla m(x_t)\Theta_0 +D_{y_1}\Theta_1 + D_{y_2}\Theta_2\right)\varphi^h\\
&+\sqrt h \left( -i\partial_ t  +\frac 12 {\rm Hess}\,  m(x_t) y\cdot y \,\Theta_0\right)\varphi^h\\
& + \left(\sum_{2\leq j\leq N} h^\frac j2 \frac 1 {j!} d^j m(x_t) [y]^j\right)\Theta_0\varphi^h + h^{\frac {N+1}2}\mathcal O\left(\||y|^{N+1}\varphi^h(t,\cdot)\|_{L^2(\R^2,\C^3)}\right).
\end{align*}
 With arguments similar to those of~\cite{BBDFLW}, solving these equations allows us to construct an edge-mode. The next result gives a condition for this program. 

\begin{lemma}\label{lem:cond_suff}
Assume that for all $t\in\R$, $\ker \widehat P_t$ is not trivial and there exists a Schwartz family of eigenfunctions of $\widehat P_t$ that generate $L^2(\R^2,\C^3)$. There exists a sequence  $(\varphi_\ell(t,\cdot))_{\ell\in\N}$ of time-dependent Schwartz functions such that for all $N\in\N$, $N\geq 2$, if \begin{equation}\label{eq:psiapp}
 \psi ^ {h,N}_{\rm app}(t,x)=\frac 1{\sqrt h}\sum_{\ell=0} ^N h^{\frac \ell 2} \varphi_\ell \left(t,\frac {x-x_t}{\sqrt h}\right),\;\; (t,x)\in \R\times \R^2,
\end{equation}
then for all $T>0$, there exists $C_{N,T}$ such that
the solution $\psi^h(t,\cdot)$ of~\eqref{eq:dirac}, with the initial data $\psi^h_0=\psi_{\rm app}^{h,N}(0,\cdot)$ satisfies 
\begin{align}
\label{eq:approx}
&\sup_{t\in[0,T]}
\| \psi ^ h(t,\cdot)-\psi ^ {h,N}_{\rm app}(t,\cdot)\|_{L^2(\R^2,\C^3)}  \leq C_{N,T}\,h^{\frac{N-1}2}.
\end{align}
\end{lemma}

Note that the estimate~\eqref{eq:approx} says no more that for all $N\geq 2$, 
\[
\sup_{t\in[0,T]}
\| \psi ^ h(t,\cdot)-\psi ^ {h,N-2}_{\rm app}(t,\cdot)\|_{L^2(\R^2,\C^3)}  \leq C_{N,T}\,h^{\frac{N-1}2}.
\]
Indeed, replacing $\psi_{\rm app}^{h,N-2}$ by $\psi_{\rm app}^{h,N}$ introduces correction terms which are of size $O(h^{\frac{N-1}2})$ in $L^2(\R^2,\C^3)$.
Alternatively, we could write that for all $N\in\N$
\[
\sup_{t\in[0,T]}
\| \psi ^ h(t,\cdot)-\psi ^ {h,N}_{\rm app}(t,\cdot)\|_{L^2(\R^2,\C^3)}  \leq C_{N,T}\,h^{\frac{N+1}2},
\]
However, our  proof leads to~\eqref{eq:approx}, which explains why we state the result in this manner.

\begin{proof}
Looking for  $\varphi^h=\sum_{0\leq \ell\leq N} h^{\frac \ell 2}\varphi_\ell$, we deduce that  the function $\varphi_0$ must be taken in the kernel of the operator~$\widehat P_t$ defined in~\eqref{def:Pt}. Then $\varphi_1$ has to satisfy 
\begin{equation}\label{eq:phi1}
\widehat P_t \varphi_1= - \left(-i\partial_ t  +\frac 12 {\rm Hess}\,  m(x_t) y\cdot y\,\Theta_0\right)\varphi_0,
\end{equation}
and the $\varphi_\ell$'s function, %$2\leq \ell\leq N$
\begin{equation}\label{eq:phi_ell}
\widehat P_t \varphi_\ell= -\left(-i\partial_ t  +\frac 12 {\rm Hess}\,  m(x_t) y\cdot y\, \Theta_0\right)\varphi_{\ell-1}
-\sum_{2\leq j\leq \ell}  \frac 1 {j!} d^j m(x_t) [y]^j\,\Theta_0\varphi_{\ell-j},\;\;2\leq\ell\leq N.
\end{equation}
The resolution of the equation for $\varphi_1$ requires that for this choice of $\varphi_0\in\ker \widehat P_t$, we have
\[
\left(-i\partial_ t  +\frac 12 {\rm Hess}\,  m(x_t) y\cdot y\,\Theta_0\right)\varphi_0\in \ker \widehat P_t^\perp.
\]
As for the equations for $\varphi_\ell$, $2\leq \ell\leq N$, we just need to be able to associate with any $g_\ell\in\mathcal S(\R^2,\C^3)$ a function $f_\ell\in \ker \widehat{P}_t $ such that 
\[
g_\ell +\left(-i\partial_ t  +\frac 12 {\rm Hess}\,  m(x_t) y\cdot y\,\Theta_0\right)f_\ell \in \ker \widehat P_t^\perp.
\]
Then, turning $\varphi_{\ell-1}$ into $\varphi_{\ell-1}-f_\ell$, $f_\ell$ associated with $g_\ell= \sum_{2\leq j\leq \ell}  \frac 1 {j!} d^j m(x_t) [y]^j\,\Theta_0\varphi_{\ell-j}$, the right-hand side of~\eqref{eq:phi_ell} belongs to $\ker\, \widehat P_t^\perp$, and we will find $\varphi_\ell$. This recursive process will produce a function $\psi^h_{\rm app}$ satisfying~\eqref{eq:psiapp} and the equation
\[
ih\partial_t \psi^h_{\rm app}-\widehat H \psi^h_{\rm app}=r_N,
\]
where, once $T>0$ is fixed, for all $t\in[0,T]$, 
\[
r_N(t,\cdot)= \mathcal{O}_T\left(h^\frac{N+1}{2}\right) \text{ in } L^2(\R^2,\C^3).
\]
Therefore, as in \cite[Lemma 3.5]{BBDFLW} for all $t\in\R$, 
\[
\frac{d}{dt}\left\|\psi^h(t,\cdot) -\psi^h_{\rm app} (t,\cdot)\right\|_{L^2(\R^2,\C^3)}\leq \frac{1}{h}\sup_{s\in[0,T]}\|r_N(s,\cdot)\|_{L^2(\R^2,\C^3)}=\mathcal{O}_T\left(h^{\frac{N-1}{2}}\right).
\]
Integrating between $0$ and $t$, we obtain 
\[
\left\|\psi^h(t,\cdot) -\psi^h_{\rm app} (t,\cdot)\right\|_{L^2(\R^2,\C^3)}
-
\left\|\psi^h_0 -\psi^h_{\rm app} (0,\cdot)\right\|_{L^2(\R^2,\C^3)}
=\mathcal{O}_T\left(h^{\frac{N-1}{2}}\right).
\]
The fact that for all $t\in\R$, $r_N(t,\cdot)$ is of size $\mathcal O(h^{\frac{N+1}2})$ in $L^2(\R^2,\C^3)$ results from the possibility of constructing the $\varphi_\ell$'s within the Schwartz class (which also requires that the kernel of $\widehat P_t$ does not reduce to $\{0\}$). 

\medskip 

We now focus on the properties of the functions $(\varphi_\ell)_{\ell\in\N}$. Let
$f_0(t,\cdot)$ be a normalized Schwartz function of $\ker\widehat P_t$. We deduce from $\| f_0(t,\cdot)\|_{L^2}=1$ the relation 
\[
{\rm Re}\,(\partial_t f_0(t,\cdot),f_0(t,\cdot))_{L^2(\R^2,\C^3)}=0.
\]
Therefore, the quantity 
\[
G(t):=\Bigl(-i\partial_ t f_0(t,\cdot) +\frac 12 {\rm Hess}\,  m(x_t) y\cdot y\,\Theta_0 f_0(t,\cdot)\;,\;f_0(t,\cdot)\Bigr)_{L^2(\R^2,\C^3)}\]
is real valued. We set  
\[
\varphi_0(t,y)={\rm e}^{-i\Lambda(t)}f_0(t,y),\;\;
\Lambda(t)=\int_0^t G(s) ds,\;\; y\in\R^2.
\]
Then, for all $t\in\R$, $\varphi_0(t,\cdot)\in \ker\,\widehat P_t$,  $\|\varphi_0(t,\cdot)\|_{L^2(\R^2,\C^3)}=1$ and we have 
\begin{align*}
&-i\partial_t \varphi_0 (t,\cdot) +\frac 12 {\rm Hess}\, m(x_t) y\cdot y\,\Theta_0 \varphi_0(t,\cdot)\\
&\;\;\;\;\; ={\rm e}^{-i\Lambda(t)} \Bigl( -G(t) f_0(t,\cdot) -i\partial_t f_0(t,\cdot) +\frac 12 {\rm Hess}\, m(x_t) y\cdot y\,\Theta_0 f_0(t,\cdot)\Bigr).
\end{align*}
Therefore, 
\begin{equation}\label{eq:ortho}
-i\partial_t \varphi_0 (t,\cdot) +\frac 12 {\rm Hess}\, m(x_t) y\cdot y\,\Theta_0 \varphi_0(t,\cdot) \in \ker \widehat P_t^\perp.
\end{equation}

At that step of the proof, we have constructed  $\varphi_0(t,\cdot)$. This function is Schwartz class (as $f_0(t,\cdot)$ is supposed to be), depends smoothly on the time variable, and it is possible to find  $u_1$ satisfying~\eqref{eq:phi1}, leading to  $\varphi_1=u_1+\lambda_1 \varphi_0$, $\lambda_1\in\C$. 
The coefficient $\lambda_1$ will be fixed.
Moreover, by the properties of the eigenfunctions of the operator $\widehat P_t$, the function $\varphi_1(t,\cdot)$ is Schwartz class. 
\smallskip 

The overall construction is a recursive process: suppose constructed 
\[ u_0=0, \;\lambda_0=1,\; u_1, \; \lambda_1,\cdots, u_{\ell-2}, \;\lambda_{\ell-2}, \; u_{\ell-1} \]
such that 
$\varphi_j=u_j+\lambda_j \varphi_0$ satisfies the equations for the indices $0\leq j\leq \ell-2$, and such that~$u_{\ell-1}$ is a solution of~\eqref{eq:phi_ell} for the index $\ell-1$. Let us construct $\lambda_{\ell-1} $ and $u_\ell$.
We
denote by $T[v]$ an operator describing the right-hand side of~\eqref{eq:phi_ell}:
for $v\in C^\infty(\R,\mathcal S(\R^2,\R^3))$, $(t,y)\in\R\times \R^2$, we set 
\[
T[v](t,y)= (-i\partial_t +\frac 12 {\rm Hess}\,  m(x_t) y\cdot y\,\Theta_0)v(t,y) + \sum_{2\leq j\leq \ell}  \frac 1 {j!} d^j m(x_t) [y]^j\,\Theta_0\varphi_{\ell-j}(t,y).
\]
For $y\in\R^2$ and $t\in\R$,
we choose
\begin{align*}
\lambda_{\ell-1}(t)&=-i\int_0^t\bigl(T[u_{\ell-1}]
(s,\cdot)
\;,\; \varphi_0(s,\cdot)\bigr)_{L^2(\R^2,\C^3)}ds,\\
\varphi_{\ell-1}(t,y)&=u_{\ell-1}(t,y)+\lambda_{\ell-1}(t) \varphi_0(t,y).
\end{align*}
Then, because of~\eqref{eq:ortho}, we have, modulo being in $\ker\,\widehat P_t^\perp$
\[
T[\varphi_{\ell-1}]
(t,\cdot)
= T[u_{\ell-1}](t,\cdot) -i\partial_t\lambda_{\ell-1}(t) \varphi_0(t,\cdot )  \;\;{\rm mod}\;\; \ker\,\widehat P_t^\perp,
\]
 whence 
$T[\varphi_{\ell-1}]\in \ker\,\widehat P_t^\perp$.
In other words, the right-hand side of equation~\eqref{eq:phi_ell} for the index $\ell$ is orthogonal to $\ker \widehat P_t$ and one can find $u_\ell$ (Schwartz class and depending smoothly on $t$) solving this equation. The construction of $\lambda_{\ell-1} $ and $u_\ell$ concludes the recursive process. This closes the proof.
\end{proof}

\subsection{Reduction to a model problem}
In this section, we look for properties of $\widehat P_t$.  
We follow the strategy of Section~\ref{subsec:strategy} and transform the operator $\widehat P_t$. First, we perform the linear change of variables $\chi_t : y\mapsto (z,s)$ described in~\eqref{def:coordzs}, and we have 
\[
\chi_t^* \widehat P_t(\chi_t)_*= \rho_t D_s +r_t z \, \Theta_0 +D_s \Theta_1(t) + D_z \Theta_2(t)
\]
where the matrices $\Theta_1(t)$ and $\Theta_2(t)$ are given by~\eqref{eq:Theta}.
Next, as explained in Section~\ref{subsec:strategy}, we  apply a  partial Fourier transform $\mathcal F_s$ on the variable $s$. Denoting by $\sigma$ this Fourier variable, we are left with the operator  
\[
\rho_t\sigma{\rm Id} +\Gamma_t(\sigma, r_tz,D_z)
\]
introduced in~\eqref{eq:gamma_hat}.
\smallskip

By Theorem~\ref{theo:model_pb}, the matrix-valued symbol $ \Gamma_t(\sigma, z,\zeta)$ is 
 unitarily equivalent via a constant unitary matrix $U_t$ to one of the matrices  $\Gamma_\bullet (\sigma,\kappa_t^{-1}(z,\delta \zeta))$, where~$\kappa_t$ is a linear symplectomorphism of $\R^2=T^*\R$, $\delta=\pm 1$, $\bullet$ stands for the index JT, PJT or $\alpha,\gamma$ with $\alpha \geq 0$ and $\gamma^2+2\alpha^2=1$:
 \begin{equation}
\label{eq:identification0}
 U_t\Gamma_t(\sigma,z,\zeta) U_t^* = \Gamma_\bullet(\sigma, u_t(z,\zeta),\delta v_t(z,\zeta)),\;\;\kappa_t^{-1}(z,\zeta)=(u_t(z,\zeta), v_t(z,\zeta)).
 \end{equation}
Since the Poisson bracket of the linear functions $v_t(r_tz,\zeta)$  and $u_t(r_tz,\zeta)$ is $r_t$, we deduce that the linear map $(z,\zeta)\mapsto \tilde\kappa_t (z,\zeta):=\bigl(r_t^{-1} u_t(r_tz,\zeta),v_t(r_tz,\zeta)\bigr)$
is a symplectomorphism and we have \begin{equation}
\label{eq:identification1}
 U_t\Gamma_t(\sigma,r_tz,\zeta) U_t^* = \Gamma_\bullet(\sigma,  u_t(r_tz,\zeta), \delta v_t(r_tz,\zeta))=
 \Gamma_\bullet(\sigma, r_t\tilde u_t(z,\zeta),\delta \tilde v_t(z,\zeta)),
   \end{equation}
   with $\tilde \kappa_t^{-1}(z,\zeta)=(\tilde u_t(z,\zeta), \tilde v_t(z,\zeta))$.
 Considering the Fourier integral operator $\Lambda_t^h$ associated with the symplectomorphism $\tilde \kappa_t$ (see Appendix~\ref{app:OIF}), we obtain the identity
\begin{align*}
 \Lambda_t^h U_t\Gamma_t(\sigma,r_t z,D_z) U_t^*(\Lambda_t^h)^*&=
 U_t\Gamma_t(\sigma,\tilde \kappa_t(z,D_z)) U_t^*\\
 &=\Gamma_\bullet (\sigma, r_t \tilde u_t(\tilde\kappa_t(z,D_z),\delta \tilde v_t(\tilde\kappa_t(z,D_z))
 \\ &=\Gamma_\bullet (\sigma, r_tz,\delta D_z).
\end{align*}
The conjugation with the Fourier transform in $s$ (which commutes with the other operators) gives 
\begin{equation}\label{eq:identification}
\Lambda_t^h U_t\chi_t^* \,\widehat P_t\, (\chi_t)_*
U_t^* (\Lambda_t^h)^* =\rho_t\,D_s+ 
 \Gamma_\bullet (D_s, r_tz,\delta D_z).
\end{equation}
Of course, $\bullet$ is JT, if $\,^t(0,0,1)$ is in the kernel of the symbol of $\Lambda_t^h U_t\chi_t^* \, \widehat P_t\,(\chi_t)_* 
U_t^* (\Lambda_t^h)^*$. 
\smallskip 

The Proposition~\ref{prop:spectral} then allows us to conclude on the choice to perform for $\rho_t$ in order to have non zero vectors in the kernel of the operator $\widetilde P_t$. We have the following alternative 
\begin{enumerate}
\item [(i)] If $\bullet$ is JT, then $0$ and $\delta$ are the only possible choices for $\rho_t$.
\item [(ii)] If $\bullet$ is PJT, then $\rho_t=\delta$ is the only possible choice. 
\item[(iii)] if $\bullet$ is $(\alpha,\gamma)$ with $\alpha\not=\frac{1}{\sqrt{2}} $ or $\delta\not= -1$, then $\rho_t=0$. If not, i.e. if $\alpha=\frac{1}{\sqrt{2}}$ and $\delta=-1$, then there is three choices for $\rho_t$: $0$, $1$ and $-1$.
\end{enumerate}
The choice of $\rho_t=0$ leads to the construction of an approximate solution to~\eqref{eq:dirac} which remains microlocalised on a single point $(x_0,0)\in\mathcal E$. Removing this possibility leads to the description of Theorem~\ref{theo:main}. It just remains to connect the model problem to properties of the matrices $\Theta_0$, $\Theta_1$ and $\Theta_2$.

\subsection{Identification of the model problem}

%%%%%%%%%%%%%%%%%%%%%
%%%%%%%%%%%%%%%%%%%%%

In order to achieve the proof of Theorem~\ref{theo:main}, it remains to verify that the model problem to which one reduces the operator $ \rho_t\sigma + \Gamma_t(\sigma,z,D_z)$ of~\eqref{eq:gamma_hat}
is determined by the relations we can observe  between the subspaces of~$\C^{3,3}$
\begin{align*}
\mathcal W:={\rm Span} _\C\left(\Theta_i,\;0\leq i\leq 2\right)\;\;{\rm and}\;\;\mathcal V:={\rm Span}_\C \left([\Theta_i,\Theta_j],\;\Theta_i,\;0\leq i,j\leq 2\right)
.
\end{align*}
Note that $\mathcal V$ and $\mathcal W$ are subspaces of the set of $3$ by $3$ trace-free matrices, with $\mathcal W\subset \mathcal V $ and by hypothesis $\dim_\C \mathcal W=3$.
\smallskip 

Setting $\Theta_0(t)=\Theta_0$ for all $t\in \R$,  by definition of the matrices $\Theta_j(t)$, $0\leq j\leq 2$ (see~\eqref{eq:Theta}),
we have 
\[
\mathcal W_t:={\rm Span} \left(\Theta_i(t),\;0\leq i\leq 2\right)=\mathcal W,\;\;\dim \mathcal W_t=3.
\]
Moreover, because of  the relations
\begin{align*}
&[\Theta_0, \Theta_1(t)]=-\cos\,\theta_t [\Theta_0,\Theta_1]-\sin\,\theta_t [\Theta_0,\Theta_2],\\
&
[\Theta_0,\Theta_2(t)]=-\sin\,\theta_t [\Theta_0,\Theta_1]+\cos\,\theta_t[\Theta_0,\Theta_2],\\
&[\Theta_1(t),\Theta_2(t)]=- [\Theta_1,\Theta_2],\;\;\forall t\in\R,
\end{align*}
we have 
\[
\mathcal V_t:=
{\rm Span} \left([\Theta_i(t),\Theta_j(t)],\;\Theta_i(t),\;0\leq i,j\leq 2\right)=\mathcal V,\;\;\forall t\in\R.
\]
As a consequence, the conditions about the sets $\mathcal V$ and $\mathcal W$ of Theorem~\ref{theo:main} can be replaced by the same conditions on $\mathcal V_t$ and $\mathcal W_t$. 
\smallskip 

 We start by analyzing these subspaces for the three model Hamiltonians. We denote them by~$\mathcal V_\bullet$ and $\mathcal W_\bullet$, indexing them in the same manner as the Hamiltonians.
We write, using~$\bullet$ for one of the indices JT, PJT or $\alpha,\gamma$,
\[ \Gamma_\bullet(\sigma,z,\zeta)=\sigma\Gamma_{0,\bullet}+z
\Gamma_{1,\bullet}+\zeta \Gamma_{2,\bullet}.
\]
The description reads as follows:

\noindent  1- {\it Jahn-Teller Hamilonian}. For $\Gamma_{\rm JT}(\sigma,z,\zeta)$, we have $\mathcal V_{\rm JT}=\mathcal W_{\rm JT}$ with the relations 
\[
[\Gamma_{0,{\rm JT}},\Gamma_{1,{\rm JT}}]=2i \Gamma_{2,{\rm JT}},
\;\;
[\Gamma_{0,{\rm JT}},\Gamma_{2,{\rm JT}}]=-2i \Gamma_{1,{\rm JT}},
\;\;
[\Gamma_{1,{\rm JT}},\Gamma_{2,{\rm JT}}]=2i \Gamma_{0,{\rm JT}}.
\]

\noindent 2- {\it Pseudo Jahn-Teller Hamiltonian}. For $\Gamma_{\rm PJT}(\sigma,z,\zeta)$, we have $\mathcal W_{\rm PJT}=\mathcal V_{\rm PJT}$ with
\[
[\Gamma_{0,{\rm PJT}},\Gamma_{1,{\rm PJT}}]=-i \Gamma_{2,{\rm PJT}}=,
\;\;
[\Gamma_{0,{\rm PJT}},\Gamma_{2,{\rm PJT}}]
={-i\Gamma_{1,{\rm PJT}}},
\;\;
[\Gamma_{1,{\rm PJT}},\Gamma_{2,{\rm PJT}}]
={i\Gamma_{0,{\rm PJT}}}.
\]

\noindent  3- {\it $\alpha,\gamma$-Hamiltonians.} For $\Gamma_{\alpha,\gamma}(\sigma,z,\zeta)$, 
we have $\mathcal W_{\alpha,\gamma}\subsetneq \mathcal V_{\alpha,\gamma}$ with $\dim \mathcal V_{\alpha,\gamma}=6$ for $\gamma\not=0$ and $\dim \mathcal V_{\frac 1{\sqrt 2},0}=5$.
Moreover, we have the  relations 
\begin{align*}
[\Gamma_{0,{{\alpha,\gamma}}},\Gamma_{1,{{\alpha,\gamma}}}]&=
\begin{pmatrix}
0 & 2i\gamma & \alpha\\
2i\gamma & 0 & -\alpha\\
-\alpha & \alpha & 0
\end{pmatrix},
\;\;
[\Gamma_{0,{\alpha,\gamma}},\Gamma_{2,{\alpha,\gamma}}]=\frac 1{\sqrt 2}\begin{pmatrix}
0 & 0 & i\\
0 & 0 & -i\\
i& -i & 0
\end{pmatrix},\\
&
[\Gamma_{1,{{\alpha,\gamma}}},\Gamma_{2,{{\alpha,\gamma}}}]=-\frac 1{\sqrt 2}\begin{pmatrix}
2i\alpha & 2i\alpha & \gamma\\
2i\alpha & 2i\alpha & -\gamma\\
-\gamma & \gamma & -4i\alpha
\end{pmatrix},
\end{align*}
Moreover, we will use that only the matrices that can be reduced to a Jahn-Teller Hamiltonian have a constant vector inside their kernel.
\smallskip 
We can now conclude the proof of Theorem~\ref{theo:main}.

\begin{proof}[End of the proof of Theorem~\ref{theo:main}]
We use equation~\eqref{eq:identification0}. We consider $a,b,c,d\in\R^4$ such that 
\[
\kappa_t^{-1}(z,\zeta)=(az+b\zeta, cz+d\zeta),\;\;\forall (z,\zeta)\in\R^2.\]
Since $\kappa_t$ is a symplectomorphism, we have $ad-bc=1$. 
Equation~\eqref{eq:identification0}  yields
\begin{align*}
&\Theta_0= U_t^*(a\Gamma_{1,\bullet}+\delta c\Gamma_{2,\bullet})U_t,\\
&\Theta_1(t)=U_t^* \Gamma_{0,\bullet}U_t,\\
&\Theta_2(t)= U_t^*(b\Gamma_{1,\bullet}+\delta d\Gamma_{2,\bullet})U_t.
\end{align*}
Therefore, we have $\mathcal V=\mathcal V_t= U_t \mathcal V_\bullet U_t^*$ and $\mathcal W=\mathcal W_t= U_t \mathcal W_\bullet U_t^*$, and 
 the following alternative holds:
\begin{enumerate}
\item[(i)] We have  $\mathcal V=\mathcal W$ and  $\bullet$ is JT if the original Hamiltonian has a smooth vector in its kernel, if not, $\bullet$ is PJT. 
\item[(ii)] We have $\mathcal W\subsetneq\mathcal V$ and $\bullet$ is $\alpha,\gamma$ for some $\alpha\in(0,1/\sqrt 2]$ and $\gamma^2=1-2\alpha^2$. Moreover, $\dim \mathcal V=6$ if $\gamma\not=0$ and $\dim \mathcal V=5$ if $\gamma=0$.
\end{enumerate}
\smallskip 

 It remains to determinate the sign of $\delta$. 
We start by observing that by~\eqref{eq:Theta} and the previous observation, the matrices $\Theta_j$ and $\Gamma_{j,\bullet}$, $0\leq j\leq 2$ are linked by the equality between matrices written by blocks
\begin{align*}
\begin{pmatrix} 
{\rm Id} & 0_{\C^{2,2}} &0_{\C^{2,2}}\\
0_{\C^{2,2}} & -\cos \theta_t {\rm Id} & -\sin \theta_t {\rm Id}\\
0_{\C^{2,2}} & -\sin\theta_t {\rm Id} & \cos \theta_t {\rm Id} 
\end{pmatrix}
\begin{pmatrix}
\Theta_0 \\ \Theta_1\\ \Theta_2
\end{pmatrix}
= \begin{pmatrix} 
 0_{\C^{2,2}} & a{\rm Id} & \delta c{\rm Id}\\
{\rm Id} & 0_{\C^{2,2}}& 0_{\C^{2,2}}\\
0_{\C^{2,2}} &  b {\rm Id} & \delta d {\rm Id} 
\end{pmatrix}
\begin{pmatrix}
U_t^*\Gamma_{0,\bullet}U_t \\ U_t^*\Gamma_{1,\bullet} U_t\\ U_t^*\Gamma_{2,\bullet} U_t
\end{pmatrix}.
\end{align*}
We obtain 
\begin{align*}
\begin{pmatrix}
U_t\Gamma_{0,\bullet} U_t^*\\ U_t\Gamma_{1,\bullet}U_t^*\\ U_t\Gamma_{2,\bullet}U_t^*
\end{pmatrix}&=
\begin{pmatrix} 
 0_{\C^{2,2}} & {\rm Id} &  0_{\C^{2,2}}\\
d{\rm Id} & 0_{\C^{2,2}}& -c{\rm Id}\\
-\delta b {\rm Id} & 0_{\C^{2,2}} & a \delta  {\rm Id} 
\end{pmatrix}
\begin{pmatrix} 
{\rm Id} & 0_{\C^{2,2}} &0_{\C^{2,2}}\\
0_{\C^{2,2}} & -\cos \theta_t {\rm Id} & -\sin \theta_t {\rm Id}\\
0_{\C^{2,2}} & -\sin\theta_t {\rm Id} & \cos \theta_t {\rm Id} 
\end{pmatrix}
\begin{pmatrix}
\Theta_0 \\ \Theta_1\\ \Theta_2
\end{pmatrix}\\
&= 
\begin{pmatrix} 
0_{\C^{2,2}} & -\cos \theta_t {\rm Id} & -\sin \theta_t {\rm Id}\\
d{\rm Id}  & c \sin\theta_t {\rm Id} & -c \cos \theta_t {\rm Id} \\
-\delta b{\rm Id}  & -a\delta \sin\theta_t {\rm Id} & a\delta \cos \theta_t {\rm Id} 
\end{pmatrix}
\begin{pmatrix}
\Theta_0 \\ \Theta_1\\ \Theta_2
\end{pmatrix}
\end{align*}
A straightforward computation gives 
\begin{align*}
U_t[\Gamma_{0,\bullet}, \Gamma_{1,\bullet}]U_t^* & =
d\cos\theta_t [\Theta_0,\Theta_1] +d\sin\theta_t [\Theta_0,\Theta_2]+ c [\Theta_1,\Theta_2].
\end{align*}
Using $ad-bc=1$ and ${\rm Tr}(A[B,A])=0$, we deduce readily
\begin{align}
\label{trace}
{\rm Tr}\left([\Gamma_{0,\bullet}, \Gamma_{1,\bullet}]\Gamma_{2,\bullet}\right) & 
= 
\delta \,  {\rm Tr} \left(\Theta_0[ \Theta_1,\Theta_2]\right).
\end{align}
Observing 
\begin{align*}
&{\rm Tr}\left(\frac 1i\bigl[\Gamma_{0,{\rm JT}}, \Gamma_{1,{\rm JT}}\bigr]\,\Gamma_{2,{\rm JT}}\right)=
2{\rm Tr}\, ({\Gamma_{2,{\rm JT}}}^2)=4,\\
&{\rm Tr}\left(\frac 1i\bigl[\Gamma_{0,{\rm PJT}}, \Gamma_{1,{\rm PJT}}\bigr]\,\Gamma_{2,{\rm PJT}}\right)=
-{\rm Tr}\,  ({\Gamma_{2,{\rm PJT}}}^2)=-2,\\
&{\rm Tr}\left(\frac 1i\bigg[\Gamma_{\frac{1}{\sqrt 2},0}, \Gamma_{\frac 1{\sqrt 2},0}\bigg]\, \Gamma_{2,\frac 1{\sqrt 2},0}\right)=
-1.
\end{align*}
We deduce that
\begin{enumerate}
\item[(i)] if $\bullet={\rm JT}$, then $\delta={\rm sgn}\left(\frac 1i{\rm Tr} \left([ \Theta_1,\Theta_2]\Theta_0\right)\right)$,
\item [(ii)] if $\bullet={\rm PJT}$ or if $\bullet$ is the pair  $ (\frac 1{\sqrt 2}, 0)$,  then 
$\delta=-{\rm sgn}\left(\frac 1i{\rm Tr} \left([ \Theta_1,\Theta_2]\Theta_0\right)\right)$.
\end{enumerate}
Here, we have used that the matrix $\frac 1i [\Gamma_{0,\bullet},\Gamma_{1,\bullet}]$ is self-adjoint, which implies that the traces above are real-valued.
This concludes the proof of Theorem~\ref{theo:main} since $\rho_t=-\delta$ both in the case $\bullet =$JT and $\bullet =$PJT.
\end{proof}

%%%%%%%%%%%%%%%%%%%%%%%%%%%
%%%%%%%%%%%%%%%%%%%%

\appendix

\section{The case of 2 by 2 matrices}\label{app:claim}

In this paragraph, we justify the claim of the introduction about the model problem~\eqref{eq:2by2}. We follow arguments which are similar to those  of the proof of Lemma~\ref{lem:coefficient}.
Assume
\[
H(x,\xi)=\Theta_0 m(x)+\Theta_1 \xi_1 +\Theta_2 \xi_2
\]
with $\Theta_j\in\C^{2,2}$ self-adjoint and trace free constant matrices. 
Assume that the eigenvalues of $H(x,\xi)$ are the functions $\sqrt{m(x)^2+|\xi|^2}$ and $-\sqrt{m(x)^2+|\xi|^2}$. Then, there exists a unitary matrix $U_1$ such that 
\[
U_1 H(x,0) U_1^*=m(x)\begin{pmatrix}1 & 0\\0&-1\end{pmatrix}, \]
which implies
\[
H(x,\xi)=\begin{pmatrix}m(x)+a_1\xi_1+a_2\xi_2 &  b_1\xi_1+ b_2\xi_2\\\bar b_1\xi_1+\bar b_2\xi_2 &-m(x)-a_1\xi_1-a_2\xi_2\end{pmatrix} \]
with $a_1,a_2\in\R$ and $b_1,b_2\in\C$.
A quick analysis of the eigenvalues of $H(x,\xi)$ yields 
\[ m(x)^2 +\xi_1^2+\xi_2^2= (m(x)+a_1\xi_1+a_2\xi_2)^2 +|b_1\xi_1+b_2\xi_2|^2,\;\;\xi_1,\xi_2\in\R,\;\;x  \in\R^2.
\]
The identification of the coefficients of this polynomial in $\xi$ and $m(x)$ gives 
\[
a_1=a_2=0,\;\; |b_1|^2=|b_2|^2=1,\;\;{\rm Re}(b_1\bar b_2)=0,
\]
whence $b_1=\e^{i\theta}$ with $\theta\in\R$ and $b_2=\pm i\e^{i\theta}$. Taking $U_2={\rm diag}(\e^{-i\theta},1)$, we have 
\[
U_2U_1 H(x,\xi) U_1^* U_2^*= \begin{pmatrix}m(x) &  \xi_1\pm i \xi_2\\\xi_1\mp i\xi_2 &-m(x)\end{pmatrix}.
\]
\smallskip

We close this section with a remark about the map from $\R^2\times \R^2$ to $\R^3$, $\Phi: (x,\xi) \mapsto (m(x),\xi_1,\xi_2) $.
Consider $x_0\in E$. Since $\nabla m(x_0)\not=0$, there exists $r_0>0$ such that $\overline B(x_0,r_0)\times \R^2\subset \Phi(\R^2\times \R^2)$.
We can then write any point $(a,b,c)\in \R^3 $ as  $
(a,b,c)= \frac {|a|}{r_0} (m(x),\xi_1,\xi_2)$ for some $(x,\xi)\in\R^2\times \R^2 $. We deduce the following fact. 
\begin{remark}\label{rem:eigenvalue_abc}
Let $\Theta_0$, $\Theta_1$ and $\Theta_2$ be three $N\times N$ matrices. We suppose that for all $(x,\xi)\in\R^2\times\R^2$, the matrix $m(x) \Theta_0+\xi_1\Theta_1+\xi_2\Theta_2$ has the eigenvalue $\sqrt{m(x)^2+|\xi|^2}$. Then, for all $(a,b,c)\in\R^3$,  the matrix $a \Theta_0+b\Theta_1+c\Theta_2$ has the eigenvalue $\sqrt{a^2+b^2+c^2}$. The same property holds for the eigenvalues~$0$ and $-\sqrt{m(x)^2+|\xi|^2}$. 
\end{remark}

\section{Linear symplectomorphism and Fourier integral operators}\label{app:OIF}

Recall that a linear symplectomorphism is a map
\[ z=(x,\xi)\in T^*\R^d=\R^d\times \R^d\mapsto Az,
\]
where $A\in\R^{2d,2d} $ is a symplectic matrix, i.e. satisfying 
$\, ^tA JA=J$ for $J=\begin{pmatrix}0 & {\rm Id}_d\\ -{\rm Id}_d & 0\end{pmatrix}$.\\
A sympletic matrix $A$ satisfies ${\rm det}A=1$. Moreover $\,^tA$ is also a symplectomorphism, and we have the property
\begin{equation}\label{matrix}
\exists M, M'\in\C^{2d,2d},\;\;A={\rm Exp}(JM){\rm Exp}(JM'), \;\; \,^tM=M,\;\;\,^tM'=M'.
\end{equation}
 We discuss this property at the end of this section. 
In the following, ${\rm op}_h(a)$ denotes a semiclassical differential operator with symbol $a$. The reader can restrict itself to the case $d=2$ and
\[
{\rm op}_h (a)= \alpha m(x) +\beta hD_{x_1} + \gamma hD_{x_2},\;\;(\alpha,\beta,\gamma)\in\R^3. 
\]
Indeed,  we  only use the result for such $a$ in this paper.

\begin{proposition}\label{prop:metaplectic}
Let $A$ be a symplectic matrix. 
There exists a unitary operator $\Lambda^h_A\in\mathcal L(L^2(\R^d))$ such that one has the exact formula 
$$\Lambda^h_A {\rm op}_h(a)(\Lambda^h_A)^* = {\rm op}_h(a(Az)).$$
\end{proposition}

The operator $\Lambda^h_A$ is  sometimes called the \emph{metaplectic transform} associated with the matrix $A$. All metaplectic transforms associated with $A$ are of the form ${\rm e}^{ i\lambda} \Lambda^h_A$ for $\lambda\in\R$.

\begin{proof}
In view of the decomposition~\eqref{matrix},  it is enough to prove the result when $A_M={\rm Exp}(JM)$ for a symmetric matrix  $M$. We will then set $\Lambda^h_A= \Lambda^h_{A_{M}}\Lambda^h_{A_{M'}}$.
In order to prove the result for the matrix $A_M$, we consider
 the propagator 
\[\Lambda_M^h(t) = \e^{\frac ih t P^h}\;\;\mbox{where}\;\; P^h= \mathrm{op}_h(p), \text{ with } p(x,\xi)=\frac12 M^t(x,\xi).^t(x,\xi).
\]
We then have for any admissible symbol $a$ 
\begin{equation}\label{egorov}
\Lambda_M^h(t) {\rm op}_h(a) \Lambda_M^h(t)^*={\rm op}_h\left( a\circ({\rm Exp}(tJM))\right).
\end{equation}
Indeed, since $p$ is quadratic, for any admissible symbol $b$, one has $\frac{i}{h}[P^h,\mathrm{op}_h(b)]=\mathrm{op}_h(\lbrace p,b\rbrace)$. Therefore, the operators
$B(t)=\Lambda_M^h(t) {\rm op}_h(a) \Lambda_M^h(t)^*$ and $C(t)={\rm op}_h\left( a\circ(\Phi_t)\right)$, where $\Phi_t={\rm Exp}(tJM)$ satisfies the same ODE. Since they coincide when $t=0$, we deduce the relation~\eqref{egorov}
We then observe that 
it is enough to take $\Lambda^h_{A_M}:=\Lambda^h_{M}(1)$ to conclude the proof.
\end{proof}

\begin{proof}[Proof of~\eqref{matrix}]
One uses the polar decomposition of $A$ : $A=QN$ with $Q$ orthogonal and $N$ symmetric positive definite. Using $A^{-1}=-J\,^t A J$, we obtain a polar decomposition of $A^{-1}$:
$$A^{-1}= -J\ N  \,^t Q J = (-J\, N J) (-J\,^tQJ).$$
Indeed, $-J\, N J$ is symmetric positive definite and $-J\,^tQJ$ orthogonal.
By the uniqueness of the polar decomposition, we deduce 
\begin{equation}\label{prop:QN}
Q^{-1}=-J\, ^t Q J\;\;\mbox{and}\;\; N^{-1}= -J\, N J.
\end{equation}
Since $N$ is symmetric positive definite, we can write $N={\rm Exp} (JM')$ with $\,^t(JM')= JM'$. Moreover, from~\eqref{prop:QN}, we deduce 
$$N^{-1}={\rm Exp} (-JM') = -JNJ=-J{\rm Exp} (JM') J={\rm Exp} (M' J),$$
whence $-JM'=M'J$ and $\, ^t M'= JM'J= M'$.\\
For constructing $M$, one diagonalizes in $\C^{2d}$ the real-valued orthogonal matrix $Q$ as a unitary matrix. One  writes 
$$Q= P \,{\rm Diag} (e^{i\theta_1},\cdots, {\rm e}^{i\theta_{2d} }) P^*, $$
with $(\theta_j)_{j\in\lbrace 1,\cdots 2d\rbrace}\subset\R$ and $P$ unitary. We define
$$JM:=P \,{\rm Diag} (i\theta_1,\cdots, i\theta_{2d}) P^*.$$
By construction $Q={\rm Exp} (JM)$. 
 We derive from $\,^tQ =Q^{-1}$ that ${\rm Exp}(-^t MJ)={\rm Exp}(-JM)$, and we deduce $\, ^t MJ=-JM.$
Then, we inject in~\eqref{prop:QN} the relations $Q^{-1}={\rm Exp} (-JM)$ and $-J\,^tQ J=-J{\rm Exp}(-^tMJ)J$, 
$${\rm Exp} (-JM)=-J{\rm Exp}(-^tMJ)J=-J{\rm Exp}(-JM)J={\rm Exp} (-MJ),$$ where we have again used $^t MJ=-JM$  and the definition of the exponential of matrices. Finally, we are left with 
$J M=MJ$ and 
$\, ^t M=-JMJ= M.$ Thus $M$ is also symmetric. 
\end{proof}

%%%%%%%%%%%%%%%%%%%%%%%%%%%%%%%%%

\bibliographystyle{plain}
\bibliography{biblio}

\end{document}